\documentclass[twoside,draft]{article}

\usepackage{amsfonts}
\usepackage{stmaryrd}
\usepackage{amssymb}
\usepackage{euscript}
\usepackage{amsthm}
\usepackage{amsmath}
\usepackage{amscd}
\usepackage{latexsym}
\usepackage{mathrsfs}
\usepackage{graphicx}
\usepackage{color}
\usepackage{dsfont}
\usepackage{bm}

\numberwithin{equation}{section}

\newtheorem{theorem}{Theorem}[section]
\newtheorem{lemma}{Lemma}[section]
\newtheorem{proposition}{Proposition}[section]
\newtheorem{corollary}{Corollary}[section]
\newtheorem{remark}{Remark}[section]

\newtheorem{definition}{Definition}[section]

\newtheorem{hyp}{Assumption}[section]
\newcommand{\bhyp }{\begin{hyp} \rm }
\newcommand{\ehyp }{\end{hyp}}

\def\proof{\noindent {\it Proof. $\, $}}
\def\endproof{\hfill $\Box$}

\newcommand{\cadlag}{RCLL}
\newcommand{\cadlags}{RCLL }

\newcommand{\wt}{\widetilde }
\newcommand{\wh}{\widehat }

\newcommand{\etab}{\eta }
\newcommand{\wtetab}{\wt{\eta }}

\def\phi{\varphi}
\newcommand{\UU }{D}

\newcommand{\alphat}{n}
\newcommand{\deltag}{\delta}

\newcommand{\dimn}{n}
\newcommand{\bfun}{b}
\newcommand{\cfun}{c}
\newcommand{\dfun}{d}

\newcommand{\qqzet}{q}
\newcommand{\qA}{\Psi^{q,\UU}}
\newcommand{\PhiA}{\Phi }

\newcommand{\LsT}{L_2({\cG}_T)}
\newcommand{\Ss}{{\mathcal{S}}^{2}}
\newcommand{\Hs}{{\mathcal{H}}^{2}(Q)}
\newcommand{\Ls}{{\mathcal{L}}^{2}(M)}

\newcommand{\Hb}{{\mathcal{H}}^{2}_{\beta}(Q)}
\newcommand{\Lb}{{\mathcal{L}}^{2}_{\beta }(M)}
\newcommand{\Sb}{{\mathcal{S}}^{2}_{\beta }}
\newcommand{\SSs}{\mathbb{S}^2}
\newcommand{\HHs}{\mathbb{H}^2}
\newcommand{\SSb}{\mathbb{S}^2_{\beta}}
\newcommand{\HHb}{\mathbb{H}^2_{\beta}}

\newcommand{\Ltg}{L^2({\cG}_T)}

\newcommand\I{\mathds{1}}

\newcommand{\E}{{\mathbb E}}

\newcommand{\cG}{\mathcal F}

\newcommand{\cE}{\mathcal E}
\newcommand{\cB}{\mathcal B}
\newcommand{\cL}{\mathcal L}
\newcommand{\cS}{\mathcal S}

\newcommand{\bff}{\mathbb F}

\newcommand{\bpp}{\mathbb P}
\newcommand{\brr}{\mathbb R}

\title{{\large \bf   
EXISTENCE, UNIQUENESS AND STRICT COMPARISON THEOREMS FOR BACKWARD STOCHASTIC DIFFERENTIAL EQUATIONS DRIVEN BY RCLL MARTINGALES} \vskip 35 pt }

\author{Tianyang Nie$\,^{a}$\footnote{The research of T. Nie and M. Rutkowski was supported by the DVC Research Bridging Support Grant {\it Pricing of American and game options in market with frictions}. The work of T. Nie was supported by the National Natural Science Foundation of China (No. 12022108,11971267,11831010,61961160732).} \ and Marek Rutkowski$\,^{b,c}$ \\ \\
\\$^{a\,}$School of Mathematics, Shandong University,\\ Jinan, Shandong
250100, China\\ \\  $^{b\,}$School of Mathematics and Statistics, University of Sydney
\\ Sydney, NSW 2006, Australia\\ \\ $^{c\,}$Faculty of Mathematics and Information Science,
Warsaw University of Technology, \\ 00-661 Warszawa, Poland \\ }

\date{\vskip 30 pt \today \vskip 25 pt}

\begin{document}

\maketitle

\begin{abstract}
Results on the existence, uniqueness and strict comparison for solutions to
a BSDE driven by a multi-dimensional RCLL martingale are established.
The goal is to develop a general multi-asset framework encompassing a wide spectrum of nonlinear financial models with jumps, including
as particular cases the setups studied by Peng and Xu \cite{PX2009,PX2010} and Dumitrescu et al. \cite{DGQS2018} who dealt with BSDEs driven by a one-dimensional Brownian motion and a purely discontinuous martingale with a single jump.


\end{abstract}

\renewcommand{\thefootnote}{\fnsymbol{footnote}}
\footnotetext{\textit{{E-mail:}} {nietianyang@sdu.edu.cn (Tianyang\ NIE); marek.rutkowski@sydney.edu.au (Marek\ RUTKOWSKI).}}

\newpage

\section{Introduction}   \label{sec1}

The goal of this paper is to reexamine and generalize results from papers by Carbone et al. \cite{CFS2008}, Dumitrescu et al. \cite{DGQS2018}, El Karoui and Huang \cite{ELH1997}, Nie and Rutkowski \cite{NR2016} and Peng and Xu \cite{PX2009,PX2010} on the backward stochastic differential equations (BSDEs) driven by multi-dimensional either continuous (\cite{NR2016}, for the extended studies on BSDEs driven by continuous martingales, the readers are referred to e.g. \cite{L2006,M2009} and the reference in  \cite{NR2016})) or RCLL (\cite{CFS2008,DGQS2018,ELH1997,PX2009,PX2010}) martingales by proving several original results and providing a comprehensive study of a particular class of BSDEs driven by a multi-dimensional RCLL martingale.
In particular, we establish the existence and uniqueness theorem for solutions to a BSDE as well as
a fairly general version of the strict comparison theorem for solutions to BSDEs, which covers both the classical comparison theorems driven by a Brownian motion (see, e.g., El Karoui et al. \cite{EPQ1997}) and their extensions to BSDEs with an additional jump martingale. The latter case was previously studied by Peng and Xu \cite{PX2009,PX2010} and Dumitrescu et al. \cite{DGQS2018} who considered BSDEs driven by a combination of the Wiener process and a single-jump compensated martingale associated with a particular event, which can be interpreted as a source of an extraneous risk (for instance, the risk of a catastrophe, an extremal weather event or a bankruptcy of a company) and is characterized by the process of a stochastic intensity of its occurrence. It is clear that the framework studied in the present work includes the case of a single-jump martingale but it covers also the case of several extraneous events occurring within a predetermined time frame and, more generally, the case of an arbitrary RCLL martingale with the predictable covariation satisfying Assumption \ref{ass2.1}.

The main practical motivation for this research is to provide an underpinning for stochastic models of nonlinear financial markets with credit risk where multi-dimensional discontinuous martingales appear in a natural way. In particular, our results encompass market models with a single default event, which were previously studied in Peng and Xu \cite{PX2009,PX2010} and Dumitrescu et al. \cite{DGQS2018}. It should be stressed that in \cite{PX2009,PX2010} and \cite{DGQS2018}, the authors examine a special class of BSDEs driven by a one-dimensional Brownian motion and a purely discontinuous martingale with a single jump of a unit size. In contrast, we study BSDEs driven by an arbitrary multi-dimensional martingale belonging to the class of multi-dimensional RCLL martingales satisfying Assumption \ref{ass2.1}. We also postulate the predictable representation property of the driving martingale, which is known to hold in some circumstances (see, e.g., Kusuoka \cite{K1999} or Jeanblanc and Le Cam \cite{JC2009}). Note, however, that this assumption can be readily relaxed in the existence and uniqueness of solutions theorem, provided that an additional orthogonal martingale term is introduced in the definition of a solution to a BSDE (see, e.g., Carbone et al. \cite{CFS2008} and El Karoui and Huang \cite{ELH1997}) and thus this possible extension is not presented here.

Due to their generality, our results are capable of covering a much broader spectrum of financial models, for instance, models with the counterparty credit risk where several default times, and thus also multi-dimensional RCLL martingales, appear in a natural way. By the same token, results from this work can be used to study market models where asset prices are governed by noise processes with jumps. Applications of our results on BSDEs and reflected BSDEs driven by RCLL martingales to the valuation, hedging and exercising of American and game options are presented in the follow-up works by Kim et al. \cite{KNR2018a,KNR2018b}. Finally, it is also worth mentioning that several classes of BSDEs driven by a Brownian motion and a martingale random measure were studied by, among others, Barles et al. \cite{BBP1995}, El Karoui et al. \cite{EMN2016}, Jeanblanc et al. \cite{JMN2012,JMN2013},  Quenez and Sulem \cite{QS2013}, Royer \cite{R2006} and Tang and Li \cite{TL1994}. We stress that their results do not cover the class of BSDEs driven by RCLL martingales considered in the present work.

After our research was completed, we learnt about an excellent recent paper by Papapantaleon et al. \cite{PPS2018}
where the authors examined well-posedness of BSDEs driven by a general martingale in a possibly stochastically discontinuous filtration.  In particular, results from \cite{PPS2018} are capable of covering BSDEs in both continuous- and discrete-time frameworks, unlike all other papers on continuous-time BSDEs driven by discontinuous martingales (and possibly also a Poisson random measure). We acknowledge that our goals were different and by far less ambitious since we focused on a particular setup, which is well-adapted to direct applications in the study of nonlinear financial markets with credit risk where the intensity-based approach to modeling of default times is ubiquitous (see, e.g., Bielecki et al. \cite{BJR2009}).
Although our setup is similar to that of Carbone et al. \cite{CFS2008}, we contend that the original contributions of the present work are nontrivial and of tangible practical importance. For the reader's convenience, we compare our methods with those from \cite{CFS2008,PPS2018} and we argue that our results cannot be
obtained from results established in either \cite{CFS2008} or \cite{PPS2018}.

Firstly, we apply directly It\^{o}'s formula and construct the contraction mapping by using the supremum norm for the component $Y$
of a solution to a BSDE (or the process $Y-D$ in a slightly more general setup of BSDE \eqref{BSDE1}). In Section 3.6 of \cite{PPS2018} and Section 3 of \cite{CFS2008},  the authors also apply It\^{o}'s formula to the appropriately weighted $L^2$-type norm (see page 44 in \cite{PPS2018} and the proof of Propostion 3.1 in \cite{CFS2008}) but they construct the contraction mapping by using this weighted $L^2$-type norm for $Y$ and thus they cannot guarantee that $Y$ is an RCLL process and, furthermore, are unable to ensure that the uniqueness of a solution to a BSDE holds in the sense of indistinguishability of stochastic processes (see, e.g., Theorem 3.23 in \cite{PPS2018} and Theorem 3.1 in \cite{CFS2008}). These manifestly technical, but nevertheless important in practical applications, shortcomings stem from the fact that the modification procedure for the classical BSDE (see, e.g., \cite{EPQ1997}) is no longer applicable when the process $\langle M \rangle$ (the predictable covariation process of the martingale $M$) is not absolutely continuous with respect to the Lebesgue measure.

Secondly, and more importantly, we give in Proposition \ref{pro5.1} an explicit solution to the linear BSDE  and subsequently use its form to prove the strict comparison theorem, which is a crucial result for financial applications, most notably, the arbitrage-free pricing of European, American and game options in nonlinear markets (see \cite{BCR2018,KNR2018a,KNR2018b}). Carbone et al. \cite{CFS2008} established the comparison theorem under some assumptions which are not easy to verify (see Theorem 2.2 and Lemma 2.2 in \cite{CFS2008}) and they studied a BSDE with the generator term involving integration w.r.t. $\langle M \rangle$, as opposed to the integration w.r.t. a process $Q$ as in the BSDE \eqref{BSDE1x}. Despite the fact that the integration w.r.t. $Q$ entails additional difficulties, we were able to establish the strict comparison theorem by first obtaining an explicit expression for a unique solution to the linear BSDE with integration w.r.t. $Q$, which is also different from \cite{CFS2008}. It should also be noticed that using the method of \cite{PPS2018} it is possible to establish the comparison theorem under rather restrictive assumptions (see Theorem 3.25 in \cite{PPS2018}) but not the strict comparison theorem obtained in the present work (see Theorem \ref{the6.1}).

The paper is organized as follows. We start by recalling some definitions from \cite{CFS2008,ELH1997,NR2016} pertaining to the backward stochastic differential equations driven by a multi-dimensional martingale. For the sake of conciseness, we refer the interested reader to \cite{NR2016} for further references and comments on various Lipschitz-type conditions for generators of BSDEs appearing in financial applications. First, in Proposition \ref{pro3.1} we obtain the {\it a priori} estimates for a solution to a BSDE. Next, in Theorem  \ref{the4.1}, we prove the existence and uniqueness of a solution to a BSDE under the assumption that its generator is uniformly $m$-Lipschitz continuous. Proposition \ref{pro5.1} furnishes an explicit representation for a unique solution to the linear BSDE. Finally, in Theorem \ref{the6.1} we establish the strict comparison property of solutions to the BSDE \eqref{BSDE1}.
Let us conclude the introduction by mentioning that as the contiuation of this work, reflected BSDEs and doubly reflected BSDEs are investigated in \cite{NR2019}.

\section{BSDEs Driven by RCLL Martingales}     \label{sec2}

We assume that we are given a filtered probability space $(\Omega,\cG,\bff ,\bpp)$ satisfying the usual conditions of right-continuity and completeness. Moreover, the initial $\sigma$-field $\cG_0$ is assumed to be trivial.  For any two processes, say $X$ and $Y$, the equality $X=Y$ means that 
$\bpp (X_t=Y_t,\,\forall\,t\in [0,T])=1$. 
By convention, we set $X_{0-}=0$ for any stochastic process $X$ and thus $\Delta X_0 := X_0-X_{0-}=X_0$ is the jump of $X$ at time 0. Finally, the integral $\int_s^t X_u\,dY_u$ is interpreted as $\int_{]s,t]}X_u\,dY_u$ where $]s,t]=\{u\in [0,T]:s<u\le t\}$.

Let $M=((M^1_t,M^2_t,\ldots,M^{\dimn}_t)^{\ast},\,t\in [0,T])$ be an $\dimn$-dimensional, square-integrable martingale defined on a filtered probability space $(\Omega,\cG,\bff,\bpp)$. Since the filtration $\bff$ is assumed to be right-continuous, an \cadlags modification of any $\bff$-martingale is known to exist and thus it is assumed that $M$ is an \cadlags process.
We denote by  $\langle M \rangle$ (resp. $[M]$) the predictable covariation process (resp. covariation process) of $M$ meaning that the process $\langle M \rangle$ (resp. $[M]$) takes values in $\brr^{\dimn \times \dimn }$ and the $(j,k)$th entry of the random matrix $\langle M \rangle_t$ (resp. $[M]$) equals $\langle M^j,M^k\rangle_t$ (resp. $[M^j,M^k]$). As in \cite{CFS2008,ELH1997,NR2016}, we henceforth work under the following standing assumption regarding the predictable covariation of $M$.

\bhyp \label{ass2.1}
There exists an $\brr^{\dimn \times \dimn}$-valued, $\bff$-predictable process $m$ and an $\bff$-adapted, continuous, nondecreasing process $Q$ with $Q_0=0$ such that, for all $t\in [0,T]$,
\begin{equation} \label{eq2.1}
\langle M\rangle_t=\int_0^t m_u m_u^{\ast}\,dQ_u.
\end{equation}
In addition, the process $Q$ is bounded: there exists a constant $C_Q$ such that $Q_t\le C_Q$ for all $t\in[0,T]$.
\ehyp

The postulated continuity of $Q$ entails that the process $\langle M \rangle$ is continuous as well. This property is known to be satisfied when $M$ is an $\bff$-quasi-left-continuous process, meaning that it does not have discontinuities at $\bff$-predictable stopping times (see Chapter VI in He et al. \cite{HWY1992} or Chapter 4 in Protter \cite{P2004}). In particular, it holds when the filtration $\bff$ is assumed to be quasi-left-continuous (as was postulated in Carbone et al. \cite{CFS2008}), that is, if $\cG_{\tau}=\cG_{\tau-}$ for every $\bff$-predictable stopping time $\tau$. Indeed, the latter property is known to be valid if and only if any uniformly integrable $\bff$-martingale is a quasi-left-continuous process. In fact, if the filtration $\bff$ is quasi-left-continuous, then $\bff$-martingales jump only at $\bff$-totally inaccessible stopping times.

The boundedness $Q$ is easy to be satisfied, for more details, see e.g. Remark 2.1 in Nie and Rutkowski \cite{NR2016} or Section 2.1 in Morlais \cite{M2009}. Let us also mention that it looks that the boundedness of $Q$ is a useful and, in some sense, a necessary condition. Indeed, as we explained in the introduction, to show that the component $Y$ of a solution to a BSDE is an RCLL process and ensure that the BSDE holds in the sense of indistinguishability of stochastic processes,  the contraction mapping based on the weighted $L^2$-type norm for $Y$ (see page 44 in \cite{PPS2018} and proof of Proposition 3.1 in \cite{CFS2008}) is not sufficient. Hence we construct the contraction mapping by using the supremum norm for $Y$ and for that purpose the boundedness of $Q$ is needed (see the proof of Theorem \ref{the4.1}).

For a positive semi-definite matrix $a$, we denote by $a^{\frac{1}{2}}$ the unique square root of $a$, that is, $a^{\frac{1}{2}}a^{\frac{1}{2}}=a$. Recall that there exists an orthogonal matrix $O$ and a diagonal matrix $b$ with nonnegative diagonal entries such that $a=O^{\ast}b O$. The square root of $b$ is also a diagonal matrix, denoted by $b^{\frac{1}{2}}$, with diagonal entries equal to square roots of diagonal entries of $b$. Then we set $a^{\frac{1}{2}}=O^{\ast}b^{\frac{1}{2}}O$. Moreover, if $a$ is positive definite, then the inverse of $a^{\frac{1}{2}}$, denoted as $a^{-\frac{1}{2}}$, is well defined. It is known that $m_u m_u^{\ast}$ is a square, positive semi-definite matrix, so that $(m_u m_u^{\ast})^{\frac{1}{2}}$ is well defined. If $m_u m_u^{\ast}$ is positive definite, then $m_u m_u^{\ast}$ is invertible, and thus we can also define $\alphat_u :=(m_u m_u^{\ast})^{-\frac{1}{2}}$.

Regarding the symmetry of the matrix-valued process $m$ appearing in \eqref{eq2.1}, observe that condition  \eqref{eq2.1} with an $\brr^{\dimn \times \dimn}$-valued, $\bff$-predictable process $m$ is equivalent to $\langle M\rangle_t=\int_0^t\bar{m}_u\bar{m}_u^{\ast}\,dQ_u$ with a symmetric $\brr^{\dimn \times \dimn}$-valued, $\bff$-predictable process $\bar{m}$ since it suffices to take $\bar{m}_u=(m_u m_u^{\ast})^{\frac{1}{2}}$. Hence we may and do assume, without loss of generality, that the process $m$ takes values in the space of symmetric matrices so that $m_u=m^{\ast}_u=(m_u m_u^{\ast})^{\frac{1}{2}}$.

Let $\|\cdot\|$ stand for the Euclidean norm in $\brr^{\dimn}$. We denote by $\Ls$ the space of all $\brr^{\dimn}$-valued, $\bff$-predictable processes $Z$ with the pseudo-norm $\|\cdot\|_{\Ls}$ given by
\begin{align*}
\|Z\|_{\Ls}^2:=\E\bigg[\int_0^T\|m_tZ_t\|^2\,dQ_t\bigg]<\infty.
\end{align*}
We also make the following standing assumption, which provides a natural underpinning for the existence of a solution the BSDE \eqref{BSDE1}.

\bhyp  \label{ass2.2}
The martingale $M$ has the predictable representation property under $\bpp$ with respect to the filtration $\bff$, in the sense that for any real-valued, square-integrable $\bff$-martingale $N$ with $N_0=0$ there exists a process $Z \in \Ls$ such that $N_t=\int_0^t Z^*_u\,dM_u$. Moreover, the uniqueness of $Z$ in
$\|\cdot\|_{\Ls}$ holds, that is, if $N_t=\int_0^t Z^*_u\,dM_u=\int_0^t \wt{Z}^*_u\,dM_u$, then $\|Z-\wt{Z}\|_{\Ls}=0$.
\ehyp

\begin{remark} \label{rem2.1}
{\rm  Assumption \ref{ass2.2} is known to be satisfied when $M$ is either a Wiener process or a compensated Poisson process. Kusuoka \cite{K1999} has shown that Assumption  \ref{ass2.2} is valid for some special choice of $\bff$ and $M$, which was made in \cite{DGQS2018,PX2009,PX2010} (see also Section 3.3 in Bielecki et al. \cite{BJR2009}, Theorem 2.1 in Jeanblanc and Le Cam \cite{JC2009} or Assumption A.1 and Example 1 in Jeanblanc et al. \cite{JMN2012,JMN2013}). If Assumption \ref{ass2.2} is relaxed, then the results of this paper can be generalized by combining the methods used here with the arguments from \cite{CFS2008,ELH1997} provided that we assume, in addition, that the filtration $\bff$ is quasi-left-continuous. Then the extended martingale representation property of $M$ holds, in the sense that any real-valued, square-integrable $\bff$-martingale $N$ can be represented as $N_t=\int_0^t Z^*_u\,dM_u+L_t$ where $L$ is a square-integrable $\bff$-martingale orthogonal to $M$.
}
\end{remark}

For a fixed horizon date $T>0$, we consider the following {\it backward stochastic differential equation} (BSDE) on $[0,T]$ with data $(g,\eta,\UU)$
\begin{equation} \label{BSDE1}
\left\{ \begin{array} [c]{ll}
dY_t=-g(t,Y_t,Z_t)\,dQ_t+Z_t^{\ast}\,dM_t+d\UU_t,\medskip\\
Y_{T}=\eta ,
\end{array} \right.
\end{equation}
where $\UU$ is a given real-valued, $\bff$-adapted process or, more explicitly, for every $t \in [0,T]$,
\begin{align} \label{BSDE1x}
Y_t=\eta+\int_t^T g(t,Y_u,Z_u)\,dQ_u-\int_t^T Z_u^{\ast}\,dM_u-(\UU_T-\UU_t)
\end{align}
where, as usual, equality \eqref{BSDE1x} is assumed to be satisfied $\bpp$-a.s.. The next definition states the measurability properties of the {\it generator} $g$.

\begin{definition} \label{def2.1}
{\rm A {\it generator} $g:\Omega \times[0,T] \times \brr \times \brr^{\dimn } \rightarrow \brr$ is a $\cG \otimes \cB ([0,T])\otimes\cB(\brr)\otimes\cB(\brr^{\dimn})$-measurable function such that the process $g(\cdot,\cdot,y,z)$ is $\bff$-progressively measurable, for every $(y,z)\in\brr \times \brr^{\dimn}$.}
\end{definition}

We denote by $\mu_Q$ the Dol\'eans measure of the process $Q$ on the space $\Omega \times [0,T]$ endowed with the product $\sigma$-algebra
$\cG_T\otimes\cB([0,T])$ so that $\mu_Q(A)=\E\big[\int_0^T\I_A (\omega,t)\,dQ_t(\omega)\big]$ for every $A\in\cG_T\otimes\cB([0,T])$.
We first recall the Lipschitz condition, which is frequently employed in the existing literature on BSDEs driven by a Brownian motion.

\begin{definition} \label{def2.2}
{\rm A generator $g$ is {\it uniformly Lipschitz continuous} if there exists a constant $L$ such that,
for $\mu_Q$-almost every $(\omega,t)$, for all $y_1,y_2\in \brr,\, z_1,z_2\in\brr^{\dimn}$,}
\begin{equation}  \label{eq2.4}
|g(t,y_1,z_1)-g(t,y_2,z_2)|\le L\big(|y_1-y_2|+\|z_1-z_2\|\big).
\end{equation}
\end{definition}

Let $\Ss$ stand for the space of all real-valued, \cadlag, $\bff$-adapted processes endowed with the norm $\|\cdot \|_{\Ss}$ given by
\begin{align*}
\|X\|_{\Ss}^2:=\E\bigg[\sup_{t\in [0,T]}X^2_t\bigg]<\infty.
\end{align*}
We denote by $\Hs$ the space of equivalence classes of all real-valued, $\bff$-progressively measurable processes $X$ with the
pseudo-norm $\|\cdot \|_{\Hs }$ given by
\begin{align*}
\|X\|_{\Hs}^2:=\E\bigg[\int_0^T X^2_t\,dQ_t\bigg]<\infty.
\end{align*}
We observe that $\Ss\subset\Hs$ and, for brevity, we denote by $\SSs$ the product space $\Ss\times\Ls$ with the following pseudo-norm,
for every $(X,Z)\in\Ss\times\Ls$,
\begin{align*}
\|(X,Z)\|_{\SSs}=\|X\|_{\Ss}+\|Z\|_{\Ls}.
\end{align*}
By the same token, we introduce the space $\HHs:=\Hs\times\Ls$ endowed with the natural pseudo-norm
\begin{align*}
\|(X,Z)\|_{\HHs}=\|X\|_{\Hs}+\|Z\|_{\Ls}.
\end{align*}
As usual, $\Ltg$ stands  for the space of all real-valued, $\cG_{T}$-measurable random variables $\eta$ such that $\|\eta\|_{\Ltg}^2=\E (\eta^2)<\infty$.

In the study of the classical BSDEs driven by a Brownian motion (also with jumps generated by a Poisson random measure), it suffices to take $Q_t=t$. Moreover, in order to ensure the existence and uniqueness of a solution $(Y,Z)$ it is common to postulate that a generator $g$ is uniformly Lipschitz continuous with respect to the variables $y$ and $z$. However, it was observed that the uniform Lipschitz condition \eqref{eq2.4} is no longer suitable when dealing with BSDEs with a single jump, which was aimed to represent the occurrence of the default event, and thus condition \eqref{eq2.4} has been modified to explicitly account for the intensity process $\lambda $. For instance, Theorem 3.1 in Peng and Xu \cite{PX2009,PX2010} and Proposition 2 in Dumitrescu et al. \cite{DGQS2018} (see also Proposition 2.3 in Dumitrescu et al. \cite{DQS2018}) hinge on the postulate that a generator satisfies the following condition
\begin{equation}  \label{eq2.5}
|g(t,y_1,z_1,k_1)-g(t,y_2,z_2,k_2)|\le L\big(|y_1-y_2|+\|z_1-z_2\|+\sqrt{\lambda_t}|k_1-k_2|\big)
\end{equation}
where the variables $z_1,z_2$ (respectively, $k_1,k_2$) correspond to the one-dimensional Brownian motion (respectively, the compensated martingale of the jump process). A generator $g(t,y,z,k)$ satisfying condition  \eqref{eq2.5} and such that the process $g_t :=g(t,0,0,0)$ belongs to $\Hs$ is called the {\it $\lambda$-admissible driver} in \cite{DGQS2018}

More generally, in existing papers concerned with BSDEs driven by a multi-dimensional martingale  (see, e.g., \cite{CFS2008,ELH1997,NR2016}) a generator is typically postulated to satisfy some kind of the $m$-Lipschitz condition where $m$ is the process introduced in Assumption \ref{ass2.1}. In particular, the BSDE driven by a multi-dimensional continuous martingale and with a uniformly $m$-Lipschitz continuous generator was studied in \cite{NR2016} where the authors introduced also certain alternative conditions, which are more adequate when the process $m$ is unbounded. In the present work, we focus on BSDEs with generator satisfying the following regularity condition.

\begin{definition} \label{def2.3}
{\rm A generator $g$ is {\it uniformly $m$-Lipschitz continuous} if there exists a constant $ \wh{L} >0$ such that, for $\mu_Q$-almost every $(\omega,t)$,
and all $y_1,y_2\in\brr,\, z_1,z_2\in\brr^{\dimn}$,}
\begin{equation} \label{eq2.6}
|g(t,y_1,z_1)-g(t,y_2,z_2)|\le \wh{L}\big(|y_1-y_2|+\|m_t(z_1-z_2)\|\big).
\end{equation}
\end{definition}

\begin{remark}  \label{rem2.2} {\rm
We note that condition \eqref{eq2.5}, which is employed in Peng and Xu \cite{PX2009,PX2010} and Dumitrescu et al. \cite{DGQS2018}, can be seen as a special case of \eqref{eq2.6}.
Indeed, suppose that $M^1:=W$ is a Brownian motion and $M^2_t:=\I_{\{\tau\le t\}}-\int_0^t \lambda_u\,du $ is the compensated martingale of the indicator process of a strictly positive $\bff$-stopping time $\tau$ with the $\bff$-intensity process $\lambda$ (note that, as in \cite{DGQS2018}, we use here the convention that the intensity process $\lambda$ is $\bff$-progressively measurable and thus it necessarily vanishes after $\tau$). Then we have  $\langle M^1 \rangle_t=\langle  W \rangle_t=t$ and $\langle M^2 \rangle_t=\int_0^t \lambda_u\,du $ since $M^2$ is a purely discontinuous martingale so that, from Theorem 4.52 in Jacod and Shiryaev \cite{JS2003}, we obtain $[M^2]_t=\sum_{u\le t}(\Delta M^2_u)^2=\I_{\{\tau\le t\}}$.  Moreover, $M^1$ is a continuous martingale and $M^2$ is a purely discontinuous martingale and thus, by Corollary 4.55 in \cite{JS2003}, we have $[M^1,M^2]_t=\langle M^1,M^2 \rangle_t=0$. Hence, if we set $Q_t=t$, then \eqref{eq2.5} can be identified with a particular instance of condition \eqref{eq2.6}, which shows that a {\it $\lambda$-admissible driver} satisfies Definition \ref{def2.3}.}
\end{remark}

We adopt the following definition of a solution to the BSDE \eqref{BSDE1}.

\begin{definition} \label{def2.4}
{\rm  A {\it solution} to the BSDE \eqref{BSDE1} with data $(g,\eta,\UU)$ is a pair of stochastic processes $(Y,Z)\in\HHs$ such that equality
\eqref{BSDE1} is satisfied pathwise, in the sense that
\begin{align*}
\bpp\bigg(Y_t=\eta+\int_t^T g(u,Y_u,Z_u)\,dQ_u-\int_t^T Z_u^{\ast}\,dM_u-(\UU_T-\UU_t),\,\forall\,t\in [0,T]\bigg)=1.
\end{align*}
We say that the {\it uniqueness of a solution to} \eqref{BSDE1} holds if for any two solutions $(Y,Z)$ and $(Y',Z')$ to \eqref{BSDE1} we have
$\|(Y,Z)-(Y',Z')\|_{\HHs}=0$.}
\end{definition}

We henceforth study BSDEs with generators satisfying Definition \ref{def2.3} and data $(g,\eta,\UU)$ satisfying the following assumption.

\bhyp \label{ass2.3}
The triplet $(g,\eta,\UU)$ is such that: \hfill \break
(i) the generator $g$ is uniformly $m$-Lipschitz continuous and the process $g(\cdot,0,0)$ belongs to $\Hs$, \hfill \break
(ii) the random variable $\eta $ belongs to $\Ltg$, \hfill \break
(iii) the process $\UU$ belongs to $\Hs$ with $\UU_{T}\in \Ltg$.
\ehyp


\begin{remark}  \label{rem2.3}
{\rm Under Assumption \ref{ass2.3}, Definition \ref{def2.4} can be restated as follows: a pair $(Y,Z)\in\HHs$ is a solution to the BSDE \eqref{BSDE1} if
the process $Y-\UU$ is \cadlags  and, for every $t\in [0,T]$, equality \eqref{BSDE1x} holds $\bpp$-a.s..
In fact, we prove in Theorem \ref{the4.1} that, under Assumptions \ref{ass2.1}--\ref{ass2.3}, there exists a solution $(Y,Z)$ to the BSDE \eqref{BSDE1} such that $(Y-\UU,Z)$ belongs to the space $\SSs$. To this end, we first establish some useful {\it a priori} estimates in Proposition \ref{pro3.1}.}
\end{remark}

\section{A Priori Estimates}  \label{sec3}

We first derive some useful {\it a priori} estimates for solutions to BSDEs.

\begin{proposition} \label{pro3.1}
Let Assumptions \ref{ass2.1}--\ref{ass2.3} be satisfied by the martingale $M$ and the triplets $(g^l,\eta^l,\UU^l)$ for $l=1,2$.
For $l=1,2$, let $(Y^l,Z^l)$ be a solution to the BSDE
\begin{equation} \label{BSDEl}
\left\{\begin{array} [c]{ll}
dY_t^l=-g^l(t,Y^l_t,Z^l_t)\,dQ_t+Z^{l,\ast}_t\,dM_t+d\UU^l_t,\medskip\\
Y_{T}^l=\eta^l,
\end{array} \right.
\end{equation}
and let $y:=Y^1-Y^2,\,z:=Z^1-Z^2,\,\UU:=\UU^1-\UU^2,\,\eta:=\eta^1-\eta^2$ and $g:=g^1-g^2$. Then there exists a constant $C\ge 0$ such that
\begin{equation} \label{eq3.2}
\|y-\UU\|_{\Ss}^2+\|y\|_{\Hs}^2+\|z\|_{\Ls}^2\le CJ
\end{equation}
where
\begin{align*}
J:=\|\eta-\UU_{T}\|^2_{\Ltg}+\|\UU\|_{\Hs}^2+\|g(t,Y^2,Z^2)\|_{\Hs}^2<\infty .
\end{align*}
\end{proposition}

\proof
We write $\wt{Y}^i=Y^i-\UU^i$, $\wt{Z}^i=Z^i$, $\wt{\eta}^i=\eta^i-\UU_T^i$, $\wt{g}^i(t,y,z):=g^i(t,y+\UU_t^i,z)$,  $\wt{y}:=\wt{Y}^1-\wt{Y}^2,\,\wt{z}:=\wt{Z}^1-\wt{Z}^2,\,\UU:=\UU^1-\UU^2$ and $\wt{\eta}:=\wt{\eta}^1-\wt{\eta}^2$.
Then we have
\begin{equation} \label{eq3.3}
\left\{ \begin{array} [c]{ll}
d\wt{y}_t=-\deltag_t\,dQ_t+\wt{z}^*_t\,dM_t,\medskip\\ \wt{y}_T=\wt{\eta},
\end{array} \right.
\end{equation}
where $\deltag_t:=\wt{g}^1(t,\wt{Y}^1_t,\wt{Z}^1_t)-\wt{g}^2(t,\wt{Y}^2_t,\wt{Z}^2_t)$.
Since $[e^{\beta Q},\wt{y}^2]=0$ by virtue of Proposition 4.49 in Chapter I of \cite{JS2003}, the It\^o integration by parts formula gives
\begin{align*}
d(e^{\beta Q_t} \wt{y}^2_t)=\wt{y}^2_t\,de^{\beta Q_t}+e^{\beta Q_t}\,d\wt{y}^2_t=\wt{y}^2_t\beta e^{\beta Q_t}\,dQ_t+e^{\beta Q_t}\,d\wt{y}^2_t
\end{align*}
where
\begin{align*}
d\wt{y}^2_t=2\wt{y}_{t-}\,d\wt{y}_t+d[\wt{y}]_t=2\wt{y}_{t-}\,d\wt{y}_t+dN_t+d\langle \wt{y} \rangle_t
\end{align*}
and the process $N:=[\wt{y}]-\langle \wt{y}\rangle $ is an $\bff$-martingale. In view of \eqref{eq3.3}, we have
\begin{align*}
\wt{y}_{t-}\,d\wt{y}_t=- \wt{y}_t \deltag_t\,dQ_t+ \wt{y}_{t-} z^*_t\, dM_t
\end{align*}
and $d\langle \wt{y} \rangle_t=\|m_t \wt{z}_t\|^2\,dQ_t$. By integrating from $t$ to $T$ and taking the conditional expectation with respect to $\cG_t$, we obtain
\begin{equation} \label{eq3.4}
\begin{split}
&e^{\beta Q_t}\wt{y}^2_t+\beta\,\E_t\bigg[\int_t^T e^{\beta Q_u}\wt{y}^2_u\,dQ_u\bigg]+\E_t\bigg[\int_t^T e^{\beta Q_u}\,\|m_u\wt{z}_u\|^2\,dQ_u\bigg]\\
&=\E_t\Big[e^{\beta Q_T}\wt{\eta}^2\Big]+2\,\E_t\bigg[\int_t^Te^{\beta Q_u}\wt{y}_u\deltag_u\,dQ_u\bigg].
\end{split}
\end{equation}
Let us denote
\begin{align*}
\wt{g}_t:=\wt{g}^1(t,\wt{Y}^2_t,\wt{Z}^2_t)-\wt{g}^2(t,\wt{Y}^2_t,\wt{Z}^2_t)=g^1(t,Y^2_t-\UU^2_t+\UU^1_t,Z^2_t)-g^2(t,Y^2_t,Z^2_t).
\end{align*}
Since $|\deltag_t|\le\wh{L}_1(|\wt{y}_t|+\|m_t\wt{z}_t\|)+|\wt{g}_t|$, we deduce from \eqref{eq3.4} that
\begin{align*}
e^{\beta Q_t}\wt{y}^2_t&+ \beta\,\E_t\bigg[\int_t^T e^{\beta Q_u}\wt{y}^2_u\,dQ_u\bigg]+\E_t\bigg[\int_t^T e^{\beta Q_u}\|m_u\wt{z}_u\|^2\,dQ_u\bigg] \\
&\le \E_t\Big[e^{\beta Q_T}\wt{\eta}^2\Big]+2\wh{L}_1\,\E_t\bigg[\int_t^T e^{\beta Q_u}\wt{y}^2_u\,dQ_u\bigg]+
2\,\E_t \bigg[\int_t^T e^{\beta Q_u}|\wt{y}_u|\big(\wh{L}_1\|m_u\wt{z}_u\|+|\wt{g}_u| \big)\,dQ_u\bigg] \\
&\le\E_t\Big[e^{\beta Q_T}\wt{\eta}^2\Big]+(2\wh{L}_1+2\wh{L}_1^2+1)\,\E_t\bigg[\int_t^T e^{\beta Q_u}\wt{y}^2_u\,dQ_u\bigg]\\
&+ \E_t \bigg[\int_t^T e^{\beta Q_u}\Big(\frac{1}{2}\|m_u\wt{z}_u\|^2+ |\wt{g}_u|^2 \big)\Big)\,dQ_u\bigg].
\end{align*}
We now take $\beta=2 \wh{L}_1+2\wh{L}_1^2+2$ and rearrange the above equation to obtain
\begin{align*}
e^{\beta Q_t} \wt{y}^2_t+\frac{1}{2}\,\E_t \bigg[\int_t^T e^{\beta Q_u}\|m_u\wt{z}_u\|^2\,dQ_u\bigg]
\le  \E\big(e^{\beta Q_T}\wt{\eta}^2\big)+\E\bigg[\int_t^T e^{\beta Q_t}|\wt{g}_t|^2\,dQ_t\bigg].
\end{align*}
 Then, recalling also that $Q$ is bounded and taking $t=0$, we obtain
the existence of a constant $C\ge 0$  (note that the value of $C$ may change from place to place) such that
\begin{align*}
\|\wt{y}\|^2_{\Ss}+\|\wt{y}\|_{\Hs}^2+\|\wt{z}\|_{\Ls}^2\le C\bigg[\E\big(\wt{\eta}^2\big)+\E\bigg(\int_0^T|\wt{g}_t|^2\,dQ_t\bigg)\bigg]
\end{align*}
where we have also used the property that $\wt{y}$ is an \cadlags  process.
Noticing that $|\wt{g}_t|^2\le 2|g(t,Y_t^2,Z_t^2)|^2+2\wh{L}_1^2|\UU_t|^2$, we get
\begin{align*}
\|\wt{y}\|^2_{\Ss}+\|\wt{y}\|_{\Hs}^2+\|\wt{z}\|_{\Ls}^2\le CJ,
\end{align*}
which in turn implies that (recalling that $\wt{y}=y-\UU$ and $\wt{z}=z$)
\begin{align*}
\| y-\UU\|^2_{\Ss}+\|y\|_{\Hs}^2+\|z\|_{\Ls}^2\le CJ.
\end{align*}
Finally, it is easy to check that $J$ is finite.
\endproof

\begin{remark} \label{rem3.1} {\rm
For later use, it is worth noting that \eqref{eq3.2} implies that if $|\eta^1-\eta^2|\le \varepsilon $ for some positive constant $\varepsilon$,
then the solutions $Y^1$ and $Y^2$ to the BSDE \eqref{BSDE1} with terminal conditions $\eta^1$ and $\eta^2$, respectively,
satisfy $|Y^1_0-Y^2_0|\le C \varepsilon $ for some constant $C$.}
\end{remark}

\section{Existence and Uniqueness Theorem for BSDEs}  \label{sec4}

To establish the existence and uniqueness of solutions to the BSDE \eqref{BSDE1}, it is useful to introduce several auxiliary normed spaces
of stochastic processes indexed by a parameter $\beta \ge 0$. First, for any fixed $\beta \ge 0$, we denote by $\Lb$ the class of all
$\brr^{\dimn}$-valued, $\bff$-predictable processes $Z$ with the pseudo-norm $\|\cdot\|_{\Lb}$ given by
\begin{align*}
\|Z\|_{\Lb}^2:=\E \bigg[\int_0^T e^{\beta Q_t}\| m_tZ_t\|^2\,dQ_t\bigg]<\infty.
\end{align*}
Hence $\Ls=\cL^2_{0}(M)$ and the pseudo-norms $\|\cdot\|_{\Ls}$ and $\|\cdot\|_{\Lb}$ are equivalent for every $\beta >0$ since the nondecreasing, nonnegative process $Q$ is bounded.

Second, we denote by $\Hb$ the space of all real-valued, $\bff$-progressively measurable processes $X$ with the pseudo-norm  $\|\cdot \|_{\Hb}$ given by
\begin{align*}
\|X\|_{\Hb}^2:=\E \bigg[\int_0^T e^{\beta Q_t}X^2_t\,dQ_t\bigg]<\infty.
\end{align*}

Third, we introduce the space $\Sb$ of all real-valued, \cadlag, $\bff$-adapted processes $X$ with the norm $\|\cdot \|_{\cS_{\beta}^2}$ given by
\begin{align*}
\|X\|_{\Sb}^2:=\E \bigg[\sup_{t\in [0,T]}\big(e^{\beta Q_t} X^2_t\big)\bigg]<\infty.
\end{align*}
It is clear that $\Ss=\cS_{0}^2$ and the norms $\|\cdot\|_{\Ss}$ and $\|\cdot\|_{\Sb}$ are equivalent for every $\beta >0$.

Finally, for a fixed $\beta \ge 0$, we denote by $\SSb$ the Banach space $(\Sb\times\Lb,\|\cdot\|_\beta)$ where the pseudo-norm $\|\cdot\|_\beta$ is given by
\begin{align*}
\|(w,v)\|^2_\beta:=\|w\|^2_{\Sb}+\|v\|^2_{\Lb}
\end{align*}
and we note that $\SSs=\mathbb{S}^2_{0}$. We stress that, in view of Assumption \ref{ass2.1}, the class $\SSb$ of stochastic processes is independent of the choice of $\beta \ge 0$ and thus the set $\SSb$ can be formally identified with $\SSs $,  although the pseudo-norm $\|\cdot\|_\beta$ depends on $\beta$ and thus it is advantageous to consider a whole family of pseudo-norms when searching for a solution to the BSDE \eqref{BSDE1} through a contraction mapping on $\SSb$ for some $\beta >0$.

Notice that equation \eqref{BSDE1x} can also be rewritten as follows
\begin{align*}
Y_t -\UU_t=\eta-\UU_T+\int_t^T\wt{g}(t,Y_u-\UU_u,Z_u)\,dQ_u-\int_t^TZ_u^{\ast}\,dM_u
\end{align*}
where we set $\wt{g}(t,y,z) :=g(t,y+\UU_t,z)$ for all $(\omega, t,y,z)\in \Omega \times [0,T]\times \brr \times \brr^{\dimn}$.
To simplify the presentation, we find it convenient to further transform the BSDE \eqref{BSDE1x} by setting
$\wt{Y}:=Y-\UU,\, \wtetab :=\etab-\UU_T$ and $\wt{Z}=Z$. Then we obtain the transformed BSDE with modified data $(\wt{g},\wt{\eta},0)$
\begin{align} \label{BSDE2}
\wt{Y}_t=\wt{\eta}+\int_t^T\wt{g}(u,\wt{Y}_u,\wt{Z}_u)\,dQ_u-\int_t^T\wt{Z}_u^{\ast}\,dM_u.
\end{align}
It is easy to see that, under Assumption \ref{ass2.3}, the random variable $\wt{\eta }$ belongs to $\Ltg$, the modified generator
$\wt{g}$ is uniformly $m$-Lipschitz continuous with a constant $\wh{L}$ and the process $\wt{g}(\cdot,0,0)$ belongs to $\Hs $.
Consistently with Definition \ref{def2.4} and Remark \ref{rem2.3}, we now search for a solution $(\wt{Y},\wt{Z})\in \HHs$ to the BSDE \eqref{BSDE2}
(in particular, $\wt{Y}$ should be an \cadlags  process). It is readily seen that the following lemma is valid.

\begin{lemma} \label{lem4.1}
Under Assumption \ref{ass2.3}, the existence and uniqueness of a solution $(\wt{Y},\wt{Z})$ to the transformed BSDE \eqref{BSDE2} is equivalent to the existence and uniqueness of a solution $(Y,Z)$ to the BSDE \eqref{BSDE1} and the respective solutions are related to each other through the mapping $(Y,Z)=(\wt{Y}+\UU, \wt{Z})$.
\end{lemma}

We are now ready to prove the existence and uniqueness theorem for solutions to the BSDE \eqref{BSDE1} with generator satisfying the uniform $m$-Lipschitz condition.

\begin{theorem} \label{the4.1}
If Assumptions \ref{ass2.1}--\ref{ass2.3} are valid, then there exists a solution $(Y,Z)\in \HHs$ to the BSDE \eqref{BSDE1} such that the process $Y-\UU$ belongs to $\Ss$. Moreover, the uniqueness of a solution $(Y,Z)$ holds in the space $\HHs$ and the uniqueness of $(Y-\UU,Z)$ holds in $\SSs $.
\end{theorem}

\begin{proof}
It is clear that both statements about uniqueness are immediate consequences of inequality \eqref{eq3.2} in Proposition \ref{pro3.1} since if $(Y,Z)$ and $(Y',Z')$ are two solutions to the BSDE \eqref{BSDE1}, then we have $g^1=g^2,\, \UU^1=\UU^2$ and $\eta^1=\eta^2$ so that \eqref{eq3.2} yields
\begin{align*}
\|Y-Y'\|_{\Ss}^2+\|Y-Y'\|_{\Hs}^2+\|Z-Z'\|_{\Ls}^2=0
\end{align*}
where we observe that  $Y-Y'=(Y-\UU)-(Y'-\UU)$. Hence it is enough to prove the first assertion.
In view of Lemma \ref{lem4.1} and since $\UU$ belongs to $\Hs$, it would be possible to assume, without loss of generality, that $\UU=0$.

In the first step, for a given pair $(w,v) \in \SSb$, we examine the following BSDE
\begin{equation} \label{eq4.2}
\left\{\begin{array} [c]{ll}
dY_t=-g_t\,dQ_t+Z_t^{\ast}\,dM_t  ,\medskip\\
Y_{T}=\eta ,
\end{array} \right.
\end{equation}
where $g_t :=g(t,w_t,v_t)$.  To establish the existence of a unique solution to \eqref{eq4.2}, it suffices to apply the postulated martingale representation property of $M$ to the square-integrable $\bff$-martingale $N_t :=\E\left[\eta+\int_0^{T}g_u\,dQ_u\big|\cG_t\right]$. Due to Assumption \ref{ass2.2}, there exists a process $Z \in \Ls$ such that, for all $t\in [0,T]$,
\begin{align*}
N_t=\E\bigg[\eta+\int_0^T g_u\,dQ_u\,\Big|\,\cG_t\bigg]=\int_0^t Z^*_u\,dM_u
\end{align*}
and the uniqueness of $Z$ in $\|\cdot \|_{\Ls}$ holds. We define the \cadlag, $\bff$-adapted process $Y$ by setting
\begin{align*}
Y_t:=\E\bigg[\eta+\int_t^T g_u\,dQ_u\,\Big|\,\cG_t\bigg]=N_t-\int_0^t g_u\,dQ_u=\int_0^t Z^*_u\,dM_u-\int_0^t g_u\,dQ_u.
\end{align*}
Then we have, for all $t \in [0,T]$,
\begin{align*}
Y_t=\eta+\int_t^T g_u\,dQ_u-\int_t^T Z_u^{\ast}\,dM_u.
\end{align*}
It is now easy to check that the pair $(Y,Z)$ belongs to $\SSb $ for any $\beta>0$ and is a solution to the BSDE \eqref{eq4.2}. Furthermore, the uniqueness of a solution to \eqref{eq4.2} in $\HHb $ (and in $\SSb$) also follows from inequality \eqref{eq3.2} in Proposition \ref{pro3.1}.

In the second step, for any $(w, v)\in \SSb$, we denote by $(Y^{w,v},Z^{w,v}) \in \SSb$ the unique solution to \eqref{eq4.2} with $g_t :=g(t,w_t,v_t)$ and
we define the mapping $\Psi : \SSb \rightarrow \SSb$ as follows: for any $(w, v) \in \SSb$, we set $\Psi(w,v)=(Y^{w,v},Z^{w,v})$. We claim that $\Psi$ is a contraction mapping on the Banach space $\SSb$, provided that $\beta $ is sufficiently large. To this end, we proceed as follows: for $l=1,2$,  we assume that $(Y^l, Z^l)$ is a solution to BSDE \eqref{eq4.2} with generator $g^l_t:=g(t,w^l_t,v^l_t)$ where $(w^l,v^l) \in \SSb$,  and we show that for an arbitrary $\beta> 0$
\begin{equation} \label{eq4.3}
\|y\|^2_{\Sb}+\|z\|^2_{\Lb}\le 1158 \beta^{-1}\wh{L}^2 (C_Q+1)\big(\|w\|_{\Sb}^2+\|v\|^2_{\Lb}\big)
\end{equation}
where $y:=Y^1-Y^2,\,z:=Z^1-Z^2,\,w:=w^1-w^2$ and $v:=v^1-v^2$.

Assuming temporarily that \eqref{eq4.3} is valid, it is immediate from \eqref{eq4.3} that $\Psi : \SSb \rightarrow \SSb$ is a contraction if we take $\beta$ sufficiently large. Then, by the Banach theorem, there exists a unique fixed point $(Y,Z)\in\SSb$ of $\Psi$. Furthermore, it is not hard to see that $Y\in\Hs$, due to the boundedness of $Q$ (when $\UU\neq 0$, we also have that $Y\in\Hs$ since $\UU\in\Hs$).  Thus the above fixed point $(Y,Z)$ is a solution to BSDE  \eqref{BSDE2} in $\SSs$ and thus also in $\HHs$.

It is now clear that to complete the proof, it suffices to establish \eqref{eq4.3}. By taking the difference of solutions to BSDEs with generator $g^l_t:=g(t,w^l_t,v^l_t)$ for $l=1,2$, we obtain
\begin{equation*}
\left\{ \begin{array} [c]{ll}
dy_t=-\deltag_t\,dQ_t+z^*_t\,dM_t,\medskip\\ y_T=0,
\end{array} \right.
\end{equation*}
where $\deltag_t:=g(t,w^1_t,v^1_t)-g(t,w^2_t, v^2_t)$. From the first step, it is clear that for an arbitrary $\beta >0$
\begin{equation}  \label{eq4.4}
\E\bigg[\sup_{t\in[0,T]}\big( e^{\beta Q_t}y_t^2\big)\bigg]=\|y\|^2_{\Sb}\le\|Y^1\|^2_{\Sb}+\|Y^2\|^2_{\Sb}<\infty
\end{equation}
since $Y^1$ and $Y^2$ belong to $\Ss $ and thus also to $\Sb$. The It\^o integration by parts formula gives
\begin{align*}
d(e^{\beta Q_t}y^2_t)=y^2_t\,de^{\beta Q_t}+e^{\beta Q_t}\,dy^2_t=y^2_t\beta e^{\beta Q_t}\,dQ_t+e^{\beta Q_t}\,dy^2_t
\end{align*}
where
\begin{align*}
dy^2_t=2 y_{t-}\,dy_t+d[y]_t=2y_{t-}\,dy_t+dN_t+d\langle y \rangle_t
\end{align*}
and $N:=[y]-\langle y \rangle$. Consequently,
\begin{equation} \label{eq4.5}
\begin{split}
d(e^{\beta Q_t} y^2_t)&=e^{\beta Q_t}\big(\beta y^2_t\,dQ_t-2y_{t-}\deltag_t\,dQ_t+2y_{t-}z^*_t\,dM_t+d[y]_t\big)\\
&=e^{\beta Q_t}\big(\beta y^2_t\, dQ_t-2y_{t-}\deltag_t\,dQ_t+2y_{t-}z^*_t\,dM_t+dN_t+d\langle y \rangle_t\big).
\end{split}
\end{equation}
To complete the derivation of \eqref{eq4.3}, we first show that
\begin{equation} \label{eq4.6}
\|z\|^2_{\Lb}\le 2\beta^{-1}\wh{L}^2(C_Q+1)\big(\|w\|_{\Sb}^2+\|v\|^2_{\Lb}\big).
\end{equation}
After integrating \eqref{eq4.5} from $t$ to $T$,  taking the expectation and using the fact that $N$ is $\bff$-martingale
and $d\langle y \rangle_t=\|m_t z_t\|^2\,dQ_t$, we obtain
\begin{align*}
&\E\bigg[e^{\beta Q_t}y^2_t+\beta\,\int_t^T e^{\beta Q_u}y^2_u\,dQ_u+\int_t^T e^{\beta Q_u}\,\|m_uz_u\|^2\,dQ_u\bigg]
\le 2\,\E\bigg[\int_t^Te^{\beta Q_u}y_u\deltag_u\,dQ_u\bigg]\\
&\le 2\wh{L}\, \E\bigg[\int_t^Te^{\beta Q_u}|y_u|(|w_u|+\|m_uv_u\|)\,dQ_u\bigg]\\
&\le 2\lambda^{-1} \wh{L}^2\,\E\bigg[\int_t^Te^{\beta Q_u}y^2_u\,dQ_u\bigg]+\lambda\,\E\bigg[\int_t^Te^{\beta Q_u}|w_u|^2\,dQ_u\bigg]
+\lambda\,\E\bigg[\int_t^Te^{\beta Q_u}\|m_uv_u\|^2\,dQ_u\bigg].
\end{align*}
Let us take $\lambda=2\beta^{-1}\wh{L}^2$. Then we get
\begin{align*}
\E \bigg[\int_0^T e^{\beta Q_u}\,\|m_uz_u\|^2\,dQ_u\bigg]\le 2\beta^{-1}\wh{L}^2\,\E\bigg[\int_0^Te^{\beta Q_u} |w_u|^2 \,dQ_u\bigg]
+2\beta^{-1}\wh{L}^2\,\E\bigg[\int_0^Te^{\beta Q_u}\|m_uv_u\|^2\,dQ_u\bigg],
\end{align*}
which gives \eqref{eq4.6} since $Q$ is a nonnegative, nondecreasing process bounded by  $C_Q$.

In the remainder of the proof, we analyze the norm $\|y\|_{\Sb}$ and we show that
\begin{equation}  \label{eq4.7}
\|y\|^2_{\Sb}\le 1156 \beta^{-1}\wh{L}^2(C_Q+1)\big(\|w\|_{\Sb}^2+\|v\|^2_{\Lb}\big).
\end{equation}
Using again \eqref{eq4.5} and noticing that the process $[y]$ is nondecreasing, we get for any $\lambda>0$ (recall that $\deltag_t:=g(t,w^1_t,v^1_t)-g(t,w^2_t, v^2_t)$
and $g$ is uniformly $m$-Lipschitz continuous)
\begin{align*}
&e^{\beta Q_t}y^2_t+\beta\,\int_t^T e^{\beta Q_u}y^2_u\,dQ_u\le 2\,\int_t^Te^{\beta Q_u}y_u\deltag_u\,dQ_u-2\int_t^Te^{\beta Q_u}y_{u-}z^*_u\,dM_u\\
&\le 2\lambda^{-1}\wh{L}^2\int_t^Te^{\beta Q_u}y^2_u\,dQ_u+\lambda\int_t^Te^{\beta Q_u}|w_u|^2\,dQ_u+\lambda\int_t^Te^{\beta Q_u}\|m_uv_u\|^2\,dQ_u
-2\int_t^Te^{\beta Q_u}y_{u-}z^*_u\,dM_u ,
\end{align*}
which in turn yields, by taking $\lambda=2\beta^{-1}\wh{L}^2$
\begin{equation} \label{eq4.8}
\begin{split}
&e^{\beta Q_t}y^2_t\le\lambda\int_t^Te^{\beta Q_u} |w_u|^2\,dQ_u+\lambda\int_t^Te^{\beta Q_u}\|m_uv_u\|^2\,dQ_u-2\int_t^Te^{\beta Q_u}y_{u-}z^*_u\,dM_u\\
&\le\lambda(C_Q+1)\left(\sup_{t\in[0,T]}\big(e^{\beta Q_t}|w_t|^2\big)+\int_t^Te^{\beta Q_u}\|m_uv_u\|^2\,dQ_u\right)-2\int_t^Te^{\beta Q_u}y_{u-}z^*_u\,dM_u .
\end{split}
\end{equation}
Let us denote $\wt{y}_u=e^{\frac{\beta}{2} Q_u}y_u$ and $\wt{M}_t :=\int_0^t e^{\frac{\beta}{2}Q_u}z^{\ast}_u\,dM_u$. The Davis inequality (see, e.g.,
Theorem 10.24 in He et al. \cite{HWY1992} or Theorem 11.5.5 in Cohen and Elliott \cite{CE2015} with $p=1$) gives
\begin{align*}
\E\left[\sup_{t\in[0,T]}\Big|\int_0^t\wt{y}_{u-}\,d\wt{M}_u\Big|\right]\le
2\sqrt{6}\,\E\left[ \bigg( \int_0^T\wt{y}_{u-}^2\,d[\wt{M}]_u\bigg)^{1/2}\right]\le
6\,\E\left[\bigg(\int_0^T\wt{y}_{u-}^2\,d[\wt{M}]_u\bigg)^{1/2}\right].
\end{align*}
Consequently, using also the elementary inequalities
\begin{align*}
\E\left[\bigg(\int_0^T\wt{y}_{u-}^2\,d[\wt{M}]_u\bigg)^{1/2}\right]
&\le\E\bigg[[\wt{M}]^{\frac{1}{2}}_T \sup_{t\in[0,T]}|\wt{y}_t|\bigg]
\le\frac{1}{48}\,\E\bigg[\sup_{t\in[0,T]}\wt{y}^2_t\bigg]+ 12\,\E\big([\wt{M}]_T\big)
\\&=\frac{1}{48}\,\E\bigg[\sup_{t\in[0,T]}\wt{y}^2_t\bigg]+12\,\E\big(\langle \wt{M}\rangle_T\big),
\end{align*}
we obtain
\begin{align*}
\E\left[\sup_{t\in[0,T]}\Big|\int_0^t\wt{y}_{u-}\,d\wt{M}_u \Big|\right]\le\frac{1}{8}\,\E\bigg[\sup_{t\in[0,T]}\wt{y}^2_t\bigg]
+72\,\E\big(\langle \wt{M}\rangle_T\big),
\end{align*}
which can be rewritten more explicitly as follows
\begin{align*}
\E\left[\sup_{t\in[0,T]}\Big|\int_0^t e^{\beta Q_u}y_{u-}z^*_u\,dM_u\Big|\right]\le
\frac{1}{8}\,\E\Big[\sup_{t\in[0,T]}\big(e^{\beta Q_t}y^2_t\big)\Big]+ 72\,\E \bigg[\int_0^T e^{\beta Q_u}\,\|m_uz_u\|^2\,dQ_u\bigg].
\end{align*}
Therefore, using also already proven inequality \eqref{eq4.6}, we find that
\begin{align*}
&\E\left[\sup_{t\in[0,T]}\Big|\int_t^Te^{\beta Q_u}y_{u-}z^*_u\,dM_u\Big|\right]\le
\frac{1}{4}\,\E\bigg[\sup_{t\in[0,T]}\big(e^{\beta Q_t}y^2_t\big)\bigg]+144\,\E\bigg[\int_0^T e^{\beta Q_u}\,\|m_uz_u\|^2\,dQ_u\bigg] \\
&\le\frac{1}{4}\,\E\bigg[\sup_{t\in[0,T]}\big(e^{\beta Q_t}y^2_t\big)\bigg]+144\lambda(C_Q+1)\big(\|w\|_{\Sb}^2+\|v\|^2_{\Lb}\big).
\end{align*}
By combining  \eqref{eq4.8} with the inequality above, we deduce that
\begin{align*}
\E\Big[ \sup_{t\in[0,T]}\big(e^{\beta Q_t}y^2_t\big)\Big]&\le\lambda(C_Q+1)\big(\|w\|_{\Sb}^2+\|v\|^2_{\Lb}\big)
+2\,\E\bigg[\sup_{t\in[0,T]}\Big|\int_t^Te^{\beta Q_u}y_{u-}z^*_u\,dM_u\Big|\bigg]\\
&\le 289\lambda(C_Q+1)\big(\|w\|_{\Sb}^2+\|v\|^2_{\Lb}\big)+\frac{1}{2}\,\E\bigg[\sup_{t\in[0,T]}\big(e^{\beta Q_t}y^2_t\big)\bigg]
\end{align*}
and thus, in view of the finiteness of the norm $\|y\|_{\Sb}$ (see \eqref{eq4.4}), we obtain
\begin{align*}
\E\bigg[\sup_{t\in[0,T]}\big(e^{\beta Q_t}y_t^2\big)\bigg]\le 578\lambda(C_Q+1)\big(\|w\|_{\Sb}^2+\|v\|^2_{\Lb}\big),
\end{align*}
which shows that \eqref{eq4.7} is satisfied (recall that $\lambda=2 \beta^{-1}\wh{L}^2$). Since inequalities \eqref{eq4.6} and \eqref{eq4.7} are now confirmed, we conclude that the desired inequality \eqref{eq4.3} is valid as well and thus the proof of the theorem is now complete.
\end{proof}

\section{Solution to the Linear BSDE}   \label{sec5}

When examining the special case of the linear BSDE, we work under two additional  postulates, which will allow us to find an explicit representation
for the component $Y$ of a unique solution $(Y,Z)$. The first one concerns the structure of the predictable covariation process $\langle M\rangle$.

\bhyp \label{ass5.1}
The matrix $m_t$ in representation \eqref{eq2.1} of the predictable covariation process $\langle M\rangle$ is invertible for all $t \in [0,T]$.
\ehyp

Assumption \ref{ass5.1} allows us to define the auxiliary function $\wh{g}(t,y,z):=g(t,y,m^{-1}_tz)$. It is readily seen that if $g$ is an arbitrary
uniformly $m$-Lipschitz continuous generator, then the function $\wh{g}$ is uniformly Lipschitz continuous since we have
\begin{equation*}
|\wh{g}(t,y_1,z_1)-\wh{g}(t,y_2,z_2)|\le\wh{L}\big(|y_1-y_2|+\|z_1-z_2\|\big).
\end{equation*}


We also impose a specific structure on the $\dimn$-dimensional martingale $M$.

\bhyp \label{ass5.2}
{\rm We assume that $M=(M^1,\dots ,M^\dimn )^*=(M^c,M^d)^{\ast}$ where the process $M^c:=(M^1,\dots ,M^k)^*$ (respectively, $M^d=(M^{k+1},\dots ,M^{\dimn})^*$) is an $\brr^k$-valued, continuous martingale (respectively, an $\brr^{\dimn-k}$-valued, purely discontinuous martingale) for some $k\in\{0,1,\dots,\dimn\}$.}
\ehyp

\begin{remark} \label{rem5.2} {\rm
We stress that Assumption \ref{ass5.2} is not directly concerned with the dynamics of asset prices, but rather with modeling of underlying sources for various kinds of financial risks. Therefore, it is not restrictive when dealing with practical applications of BSDEs driven by an RCLL martingale to problems arising in financial mathematics (see, in particular, Dumitrescu et al. \cite{DQS2018} or Kim et al. \cite{KNR2018a,KNR2018b} for recent studies of American and game options in nonlinear markets).}
\end{remark}

Let us denote $\alphat_t:=m_t^{-1}$ so that $\alphat_t^{-1}=m_t=m_t^{\ast}$. Under Assumption \ref{ass5.2}, we have that $\langle M^i,M^j \rangle=0$ for all $i,j$ such that $i\le k$ and $j\ge k+1$. Therefore, we may and do assume that $m$ is a symmetric block matrix and thus the notation $\alphat^c_t=(m^c_t)^{-1}$ and $\alphat^d_t=(m^d_t)^{-1}$ where $m^c_t$ and $m^d_t$ are symmetric random matrices of dimensions $k$ and $\dimn -k$, respectively, can be easily justified. By the same token, for any $z=(z_1,\dots ,z_\dimn)^*\in\brr^n$ we write $z=(z^c,z^d)^*$ where $z^c=(z_1,\dots ,z_k)$ and $z^d=(z_{k+1},\dots ,z_\dimn)$. An analogous notational convention can be applied to the process $Z$, so that we may write $Z=(Z^c,Z^d)^*$ and to other stochastic processes of interest.

We now examine the linear BSDE
\begin{equation} \label{linBSDE}
\left\{ \begin{array} [c]{ll}
dY_t=-\big(a_tY_t+\bfun_t+\cfun^*_t m^{c}_tZ^c_t+\dfun^*_t m^{d}_tZ^d_t\big)\,dQ_t+Z_t^{\ast}\,dM_t+d\UU_t,\medskip\\
Y_{T}=\eta ,
\end{array} \right.
\end{equation}
and we work under the following variant of Assumption \ref{ass2.3}.

\bhyp \label{ass5.3}
The triplet $(g,\eta ,\UU)$ is such that: \hfill \break
(i) the generator $g$ equals $g(t,y,z)=a_ty+\bfun_t+\cfun^*_t m^{c}_tz^c+\dfun^*_tm^{d}_tz^d$ where the real-valued processes $a,\cfun$ and $\dfun$  are bounded and
$\bfun$ belongs to $\Hs$. \hfill \break
(ii) the random variable $\eta $ belongs to $\Ltg$, \hfill \break
(iii) $\UU$ is an \cadlags  process of finite variation and belongs to $\Hs$; furthermore,
$|\UU|_{T}\in \Ltg$ where $|\UU|=(|\UU|_t,\,t\in [0,T])$ is the total variation process of $\UU$. 
\ehyp

\begin{remark} \label{rem5.3}
{\rm Notice that the total variation of $\UU$ satisfies
\begin{align*}
|\UU|_T \geq \sum_{u\le t}|\Delta\UU_u|\geq\bigg(\sum_{u\le t}(\Delta\UU_u)^2\bigg)^{1/2}.
\end{align*}
Hence we deduce from part (iii) in Assumption \ref{ass5.3} that $\E \big(\sum_{u\le t}(\Delta \UU_u)^2\big)<\infty $.}
\end{remark}

Since the processes $a,c,d$ are assumed to be bounded, the generator $g(t,y,z)=a_ty+ \bfun_t+ \cfun^*_t m^{c}_tz^c+ \dfun^*_tm^{d}_tz^d$ is
uniformly $m$-Lipschitz continuous and the process $g(\cdot,0,0)=c$ belongs to $\Hs $. Therefore, we may apply Theorem  \ref{the4.1}
to ensure that the BSDE \eqref{linBSDE} has a unique solution $(Y,Z) \in \Ss$ provided that Assumptions \ref{ass2.1}-\ref{ass2.2} and \ref{ass5.1}-\ref{ass5.3} are satisfied. Our goal is to find an explicit representation for the process $Y$, which is convenient to establish the comparison property.

Before stating the next result, we need to introduce some notation. For an arbitrary $\brr^{\dimn-k}$-valued process $\dfun$, which is assumed to be bounded and $\bff$-predictable, we define the \cadlag, $\bff$-adapted process $\qqzet$ by
\begin{equation} \label{eq5.2}
\qqzet_t:=\cE_t(\wh{Q})\cE_t(\wh{M})=\cE_t(\wh{Q}+\wh{M})
\end{equation}
where $\wh{Q}$ is a continuous, bounded process of finite variation given by $\wh{Q}_t:=\int_0^t a_u\,dQ_u$ and $\wh{M}$
is an $\bff$-martingale defined by
\begin{equation} \label{eq5.3}
\wh{M}_t:=\int_0^t\cfun^*_u\alphat^c_u\,dM^c_u+\int_0^t\dfun^{\ast}_u\alphat^d_u\,dM^d_u=\wh{M}^c_t+\wh{M}^d_t.
\end{equation}
Here $\cE(\wh{M})$ denotes the Dol\'eans exponential given by the following expression
\begin{align*}
\cE_t(\wh{M})=\exp\Big(\wh{M}_t-\frac{1}{2}\langle\wh{M}^c\rangle_t\Big)\prod_{0<u\le t}\big(1+\Delta\wh{M}_u\big)e^{-\Delta\wh{M}_u}.
\end{align*}
For brevity, we denote
\begin{equation} \label{eq5.4}
\qA_t:=\int_0^t\qqzet_{u-}\,d\UU_u+\langle \qqzet,\UU\rangle_t .
\end{equation}

We are ready to derive an explicit representation for the process $Y$ in the unique solution to the linear BSDE.

\begin{proposition} \label{pro5.1}
If Assumptions \ref{ass2.1}-\ref{ass2.2} and \ref{ass5.1}-\ref{ass5.3} hold,
then the linear BSDE \eqref{linBSDE} has a unique solution $(Y,Z)\in \Ss $ and the process $Y$ satisfies
\begin{equation} \label{eq5.10}
\qqzet_t Y_t=\E\bigg[\qqzet_{T}\eta+\int_t^T\qqzet_u\bfun_u\,dQ_u-\big(\qA_T-\qA_t\big)\,\Big|\,\cG_t\bigg]
\end{equation}
where $\qqzet$ and $\qA$ are given by \eqref{eq5.2} and \eqref{eq5.4}, respectively.
\end{proposition}

\proof
Since $\wh{Q}$ is an $\bff$-adapted, bounded, continuous process of finite variation and $\wh{M}$ is an $\bff$-martingale given by \eqref{eq5.3},
it is obvious that $[\wh{Q},\wh{M}]=0$. Recall that $\cE(\wh{M})$ is given by
\begin{align*}
\cE_t(\wh{M})=\exp\Big(\wh{M}_t-\frac{1}{2}\langle\wh{M}^c\rangle_t\Big)\prod_{0<u\le t}\big(1+\Delta\wh{M}_u\big)e^{-\Delta\wh{M}_u}
\end{align*}
where $\wh{M}^c_t=\int_0^t\cfun^*_u \alphat^c_u\,dM^c_u$ and $\Delta\wh{M}_u:=\wh{M}_u-\wh{M}_{u-}=\wh{M}^d_u-\wh{M}^d_{u-}=\dfun^*_u\alphat^d_u\Delta M^d_u$.
Simple computations show that $\qqzet$ can be represented as follows
\begin{align*}
\qqzet_t=\exp\left(\int_0^t a_u\,dQ_u\right)\exp\left(\wh{M}^c_t-\frac{1}{2}\int_0^t\|\cfun_u\|^2\,dQ_u\right)
\exp\big(\wh{M}^d_t\big)\prod_{0< u\le t}\big(1+\Delta \wh{M}^d_u\big)e^{-\Delta \wh{M}^d_u}.
\end{align*}
In view of \eqref{eq5.2} and \eqref{eq5.3}, we have that
\begin{align*}
d\qqzet_t=\qqzet_{t-}\,d(\wh{Q}_t+\wh{M}_t)=\qqzet_{t-}\,d\wh{Q}_t+ \qqzet_{t-}\,d\wh{M}_t
=\qqzet_{t-} \big(a_t\,dQ_t+\cfun^*_t\alphat^c_t\,dM^c_t+\dfun^{\ast}_t\alphat^d_t\,dM^d_t\big)
\end{align*}
and thus, using \eqref{linBSDE}, we obtain
$$
d\langle\qqzet,Y\rangle_t=\qqzet_t\big(\cfun^*_t m^{c}_tZ^c_t+\dfun^{\ast}_t m^{d}_tZ^d_t\big)\,dQ_t+\langle q,\UU \rangle_t
$$
where we denote $Z=(Z^c,Z^d)^*$. Since $\UU$ is an \cadlags  process of finite variation, and we have
\[
[\qqzet,\UU]_t=\sum_{u\le t}\Delta\qqzet_u\Delta\UU_u
\]
where $\Delta\qqzet_t=\qqzet_t-\qqzet_{t-}$ and $\Delta\UU_t=\UU_t-\UU_{t-}$.  Using the It\^{o} integration by parts formula, we obtain
\begin{align*}
d(\qqzet_tY_t)&=\qqzet_{t-}\,dY_t+Y_{t-}\,d\qqzet_t+d[q,Y]_t \\
&=-\qqzet_t\big(a_tY_t+\bfun_t+\cfun^*_tm^{c}_tZ^c_t+\dfun^*_tm^d_tZ^d_t\big)\,dQ_t+\qqzet_{t-}Z_t^{\ast}\,dM_t+\qqzet_{t-}\,d\UU_t\\
&\quad\mbox{}+\qqzet_tY_ta_t\,dQ_t+\qqzet_{t-}Y_{t-}\big(\cfun^*_t\alphat^c_t\,dM^c_t+\dfun^{\ast}_t\alphat^d_t\,dM^d_t\big)
+d\langle q,Y \rangle_t+d[q,Y]_t-d\langle q,Y\rangle_t \\
&=d\wt{M}_t-\qqzet_t\big(\bfun_t+\cfun^*_tm^{c}_tZ^c_t+ \dfun^*_tm^d_tZ^d_t\big)\,dQ_t
+ \qqzet_t\big(\cfun^*_t m^{c}_tZ^c_t+\dfun^{\ast}_t m^{d}_t Z^d_t\big)\,dQ_t+\qqzet_{t-}\,d\UU_t+d\langle q,\UU \rangle_t \\
&=d\wt{M}_t-\qqzet_t\bfun_t\,dQ_t+d\qA_t
\end{align*}
where the \cadlags  process $\wt{M}$ is defined by
\begin{equation*}
d\wt{M}_t=\qqzet_{t-}\big(Z_t^{\ast}\,dM_t+Y_{t-}\cfun^*_t\alphat^c_t\,dM^c_t
+Y_{t-}\dfun^{\ast}_t\alphat^d_t \,dM^d_t \big)+d[\qqzet,Y]_t-d\langle \qqzet,Y\rangle_t
\end{equation*}
where $\wt{M}_0:=\qqzet_0Y_0=Y_0$ and the process $N:=[\qqzet,Y]-\langle\qqzet,Y\rangle$ is an $\bff$-local martingale with $N_0=0$. We also have that
\begin{align*}
\wt{M}_T-\wt{M}_t=\qqzet_{T}\eta-\qqzet_tY_t+\int_t^T\qqzet_ub_u\,dQ_u-(\qA_T-\qA_t)
\end{align*}
and thus, in view of Lemma \ref{lem5.1}, the process $\wt{M}$ is an $\bff$-martingale. Hence, after taking the conditional expectation with respect
to $\cG_t$, we obtain \eqref{eq5.10}.
\endproof

The following lemma is required for the proof of Proposition \ref{pro5.1} and Theorem \ref{the6.1}.

\begin{lemma} \label{lem5.1}
If Assumptions \ref{ass2.1}-\ref{ass2.2} and \ref{ass5.1}-\ref{ass5.3} are satisfied, then the process $\wt{M}$ given by
\begin{equation} \label{eq5.5}
\wt{M}_t=\qqzet_tY_t-\qA_t+\int_0^t\qqzet_u\bfun_u\,dQ_u,
\end{equation}
is an $\bff$-martingale.
\end{lemma}

\proof
We claim that
\begin{equation} \label{eq5.6}
\E \bigg[\sup_{t\in[0,T]}|\wt{M}_t|\bigg]<\infty,
\end{equation}
which is known to imply that $\wt{M}$ is a uniformly integrable $\bff$-martingale. Our goal is thus to establish \eqref{eq5.6}. For this purpose,
we start by examining the process $\qqzet$. Since $\cfun$ and $\dfun$ are bounded, we have, for some constant $C$,
\begin{equation} \label{eq5.7}
\langle \wh{M}\rangle_T=\int_0^T\big(|\cfun_t|^2+|\dfun_t|^2\big)\,dQ_t\le C<\infty.
\end{equation}
We will deduce from \eqref{eq5.2} that
\begin{equation} \label{eq5.8}
\|\qqzet\|^2_{\Ss}=\E\bigg[\sup_{t\in[0,T]}\qqzet_t^2\bigg]<\infty.
\end{equation}
To this end, we first observe that \eqref{eq5.2} gives $\qqzet_t=\cE_t (\wh{Q})\cE_t (\wh{M})$ where the process $\cE(\wh{Q})$ is bounded since $a$
and $Q$ are bounded. To establish \eqref{eq5.8}, it thus suffices to show that
\begin{equation} \label{eq5.9}
\|\cE(\wh{M})\|^2_{\Ss}=\E\bigg[\sup_{t\in[0,T]}\cE_t^2(\wh{M})\bigg]<\infty.
\end{equation}
Observe that
\begin{align*}
d\cE_t(\wh{M})=\cE_{t-}(\wh{M})\,d\wh{M}_t=\cE_{t-}(\wh{M})\,d\wh{M}^c_t+\cE_{t-}(\wh{M})\,d\wh{M}^d_t
\end{align*}
and thus
\begin{align*}
d[\cE(\wh{M})]_t&=\cE_{t-}^2(\wh{M})\,d\langle\wh{M}^c\rangle_t+\cE_{t-}^2(\wh{M})\,d[\wh{M}^d]_t\\
&=\cE_{t-}^2(\wh{M})\big(d\langle\wh{M}^c\rangle_t+d\langle\wh{M}^d\rangle_t\big)
+\cE_{t-}^2(\wh{M})\big(d[\wh{M}^d]_t-d\langle\wh{M}^d\rangle_t\big),
\end{align*}
which in turn yields
\begin{align*}
d\cE_t^2(\wh{M})&=2\cE_{t-}(\wh{M})\,d\cE_t(\wh{M})+d[\cE(\wh{M})]_t\\
&=\cE_{t-}^2(\wh{M})\big(d\langle\wh{M}^c\rangle_t+d\langle\wh{M}^d\rangle_t\big)
+\cE_{t-}^2(\wh{M})\big(2\,d\wh{M}_t+d[\wh{M}^d]_t-d\langle\wh{M}^d\rangle_t\big).
\end{align*}
Recalling that $\cE_{0}(\wh{M})=1$, we obtain $\cE_t^2 (\wh{M})=\gamma_t\exp\big(\langle\wh{M}^c\rangle_t+\langle\wh{M}^d\rangle_t\big)$
where $\gamma_0=1$ and
\begin{align*}
d\gamma_t=\gamma_{t-}\Big(2\,d\wh{M}_t+d[\wh{M}^d]_t-d\langle \wh{M}^d\rangle_t\Big).
\end{align*}
From the last two equations, it follows that $\gamma$ is a non-negative local martingale and hence a supermartingale,
so that $\E[\gamma_t]\le\E[\gamma_0]=1$ for all $t\in[0,T]$. It is clear from \eqref{eq5.7} that there exists a constant $C$ such that, for all $t \in [0,T]$
\begin{align*}
\exp\big(\langle\wh{M}^c\rangle_t+\langle\wh{M}^d\rangle_t\big)\le C.
\end{align*}
Consequently, it follows that $\cE_t^2 (\wh{M})\le C\gamma_t$  for all $t \in [0,T]$. Since $\cE^2 (\wh{M})$ and $\gamma$ are \cadlags processes,  we deduce that
$\cE_{\tau}^2 (\wh{M})\leq C\gamma_{\tau}$ for all $\tau\in\mathcal{T}$ where $\mathcal{T}$ is the set of  all stopping times taking values in $[0,T]$.
By applying the optional stopping theorem to $\gamma$, we get
\begin{align*}
\E\Big[\cE_{\tau}^2 (\wh{M})\Big]\leq C\, \E\big[\gamma_{\tau}\big]\leq C, \text{ for all } \tau\in\mathcal{T}.
\end{align*}
Hence the family $\{\cE_{\tau}(\wh{M})\}_{\tau\in\mathcal{T}}$ is uniformly integrable, which implies that $\cE(\wh{M})$
is a nonnegative \cadlags  martingale. From Doob's inequality, we obtain
\begin{align*}
\E\bigg[\sup_{t\in[0,T]}\cE_{t}^2 (\wh{M})\bigg]\leq 4\, \E\big[\mathcal{E}_{T}^2(\wh{M})\big]\leq 4C,
\end{align*}
which shows that \eqref{eq5.9} is valid and thus \eqref{eq5.8} holds as well.

We are now ready to analyze the right-hand side of \eqref{eq5.5}. It is clear that
\begin{align*}
\big|\qqzet_tY_t-\qA_t\big|\le |\qqzet_t||Y_t-\UU_t|+|\qqzet_t||\UU_t|+(|\UU_0|+|\UU_T|)\sup_{t\in [0,T]}\qqzet_t+\sup_{t\in [0,T]}|\langle\qqzet,\UU\rangle_t|.
\end{align*}
Since $\UU$ is assumed to be a process of finite variation with the total variation process  $|\UU|$ such that $\E(|\UU|^2_T)<\infty$, we have $\sup_{\,t\in [0,T]} \UU^2_t\le 2(\UU^2_0+|\UU|^2_T)<\infty$ and thus $\UU$ belongs to $\Ss$. Of course, the process $Y-\UU$ is in $\Ss$ as well (see Theorem \ref{the4.1}). Using the Kunita-Watanabe inequality (see, for instance, Theorem 8.3 in He et al. \cite{HWY1992}), we obtain, for all $t\in [0,T]$,
\begin{align*}
|\langle\qqzet,\UU\rangle_t|\le(\langle\qqzet\rangle_t\langle\UU\rangle_t)^{1/2}\le 2(\langle\qqzet\rangle_t+\langle\UU\rangle_t)
\le 2(\langle\qqzet\rangle_T+\langle\UU\rangle_T)
\end{align*}
and thus
\begin{align*}
\E \bigg[\sup_{t\in [0,T]}|\langle\qqzet,\UU\rangle_t|\bigg]\le 2\big(\E(\langle\qqzet\rangle_T)+\E(\langle\UU\rangle_T)\big).
\end{align*}
Since $\UU$ is a process of finite variation, we have $[\UU ]_t=\sum_{u\le t}(\Delta \UU_u)^2$ and, from part (iii) in Assumption \ref{ass5.3},
we infer that $\E ([\UU]_T)<\infty$ (see Remark \ref{rem5.3}), which in turn implies that $\E (\langle\UU\rangle_T)<\infty$. Similarly, since
\begin{align*}
d\qqzet_t=d(\cE_t(\wh{Q})\cE_t(\wh{M}))=\cE_t(\wh{Q})\,d\cE_t(\wh{M})+\cE_{t-}(\wh{M})\,d\cE_t(\wh{Q})
\end{align*}
and $\langle\cE(\wh{Q})\rangle=\langle\cE(\wh{Q}),\cE(\wh{M})\rangle=0$, using the boundedness of the processes $\cE(\wh{Q})$ and $\langle\wh{M}\rangle$ (see \eqref{eq5.7}),
we obtain
\begin{align*}
\E (\langle\qqzet\rangle_T)=\E \bigg[\int_0^T\cE^2_t(\wh{Q})\cE^2_{t-}(\wh{M})\,d\langle\wh{M}\rangle_t\bigg]\le C\|\cE(\wh{M})\|^2_{\Ss}.
\end{align*}
We conclude that
\begin{align*}
\E\bigg[\sup_{t\in [0,T]}\big|\qqzet_tY_t-\qA_t \big|\bigg]\le C\big[\|\qqzet\|_{\Ss}(\|y-\UU\|_{\Ss}+\|\UU\|_{\Ss})
+\E(\langle\UU\rangle_T)+\|\cE(\wh{M})\|^2_{\Ss}\big]<\infty
\end{align*}
where we have also used \eqref{eq5.8} and \eqref{eq5.9}.

Now let us study the integral $\int_0^t \qqzet_u\bfun_u\,dQ_u$. Using the boundedness of $Q$, we obtain the existence of constant $C$ which may very line by line such that
\begin{align*}
\E\bigg[\sup_{t\in[0,T]}\Big|\int_0^t\qqzet_u\bfun_u\,dQ_u\Big|\bigg]
&\le\E\bigg[\sup_{t\in[0,T]}\bigg(\int_0^t\qqzet^2_u\bfun^2_u\,dQ_u\bigg)^{1/2}Q_{T}^{1/2}\bigg]
\le C\,\E\bigg[ \bigg(\int_0^T\qqzet^2_u\bfun^2_u\,dQ_u\bigg)^{1/2}\bigg]\\
&\le C\,\E\bigg[\sup_{t\in[0,T]}\qqzet_t\bigg(\int_0^T\bfun^2_u\,dQ_u\bigg)^{1/2}\bigg]\le 2C(\|\qqzet\|_{\Ss}+\|\bfun\|_{\Hs})<\infty
\end{align*}
where we have also used \eqref{eq5.8} in the last inequality. We have thus shown that \eqref{eq5.6} holds.
\endproof

Let us examine some consequences of Proposition \ref{pro5.1}.
If we assume, in addition, that $\Delta\wh{M}_t=\Delta\wh{M}^d_t>-1$ for all $t\in [0,T]$, then the process $\qqzet$ given by \eqref{eq5.2} is strictly positive
and thus \eqref{eq5.10} yields the following representation for $Y$
\begin{align*}
Y_t=\E\bigg[\Gamma_{t,T}\eta+\int_t^T\Gamma_{t,u}\bfun_u\,dQ_u-\int_t^T\Gamma_{t,u-}\,d\UU_u- \qqzet_t^{-1}\big(\langle\qqzet,\UU\rangle_T-\langle
\qqzet,\UU\rangle_t\big)\,\Big|\,\cG_t\bigg]
\end{align*}
where we denote $\Gamma_{t,u}=\qqzet_t^{-1}\qqzet_u$ for all $0\le t\le u\le T$.

In particular, if the martingale $M$, and thus also the process $\qqzet$, are continuous, we obtain the following corollary, which furnishes a minor extension of the classical representation of a solution to a linear BSDE driven by the Brownian motion (see, for instance,
Proposition 2.2 in El Karoui et al. \cite{EPQ1997}). Since $M$ is continuous, it is clear that $m=m^c$ (hence $\alphat=\alphat^c$) and thus there is no need to introduce the processes $m^d,\alphat^d$ and $d$ in \eqref{linBSDE} and  \eqref{eq5.3}.  Therefore, these equations reduce to
\begin{equation} \label{eq5.11}
\left\{ \begin{array} [c]{ll}
dY_t=-\big(a_tY_t+\bfun_t+\cfun^*_t m_tZ_t \big)\,dQ_t+Z_t^{\ast}\,dM_t+d\UU_t,\medskip\\
Y_{T}=\eta ,
\end{array} \right.
\end{equation}
and
\begin{equation} \label{eq5.12}
\wh{M}_t:=\int_0^t\cfun^*_u\alphat_u\,dM_u .
\end{equation}

\begin{corollary} \label{cor6.1}
Let Assumptions \ref{ass2.1}-\ref{ass2.2} and \ref{ass5.1}-\ref{ass5.3} hold so that the linear BSDE \eqref{eq5.11} has a unique solution
$(Y,Z)\in\Ss$. If the $\bff$-martingale $M$ is continuous, then $Y$ equals
\begin{align*}
Y_t=\E\bigg[\Gamma_{t,T}\eta+\int_t^T\Gamma_{t,u}\bfun_u\,dQ_u-\int_t^T\Gamma_{t,u}\,d\UU_u\,\Big|\,\cG_t\bigg]
\end{align*}
where $\Gamma_{t,u}=\qqzet_t^{-1}\qqzet_u$ with $\qqzet_t=\cE_t(\wh{Q})\cE_t(\wh{M})$ where $\wh{Q}_t=\int_0^t a_u\,dQ_u$ and $\wh{M}_t=\int_0^t\cfun^*_u\alphat_u\,dM_u$.
\end{corollary}

\section{Strict Comparison Theorem}   \label{sec6}

An important feature of solutions to a BSDE is the so-called (strict) {\it comparison property} (see the statement of Theorem  \ref{the6.1}), which underpins several important applications of BSDEs to arbitrage-free pricing. Therefore, our final goal is to establish the comparison theorem for the BSDE \eqref{BSDEl} within the setup given by Assumptions \ref{ass2.1}, \ref{ass2.2}, \ref{ass5.1} and \ref{ass5.2}. To prove the comparison property of solutions to the BSDE \eqref{BSDEl} by the well-known method of linearization, we define the following auxiliary processes
$$
a_t:=\frac{g^1(t,Y^1_t,Z^1_t)-g^1(t,Y^2_t,Z^1_t)}{Y^1_t-Y^2_t}\,\I_{\{Y^1_t\neq Y^2_t\}},
$$
\begin{align*}
c^1_t:=\frac{1}{\wt{Z}^1_{t,1}-\wt{Z}^2_{t,1}}\Big[ \wh{g}^1(t,Y^2_t,\wt{Z}^1_{t,1},\wt{Z}^1_{t,2}\ldots,\ldots,\wt{Z}^1_{t,d})
- \wh{g}^1(t,Y^2_t,\wt{Z}^2_{t,1},\wt{Z}^1_{t,2}\ldots,\ldots,\wt{Z}^1_{t,d})\Big]\I_{\wt{D}_1},
\end{align*}
and, for every $j=2,3,\ldots,k$,
\begin{align*}
c^j_t&:=\frac{1}{\wt{Z}^1_{t,j}-\wt{Z}^2_{t,j}}\Big[ \wh{g}^1(t,Y^2_t,\wt{Z}^2_{t,1},\ldots,
\wt{Z}^2_{t,j-1},\wt{Z}^1_{t,j},\wt{Z}^1_{t,j+1},\ldots,\wt{Z}^1_{t,d})\\ &-\wh{g}^1(t,Y^2_t,\wt{Z}^2_{t,1},\ldots,\wt{Z}^2_{t,j-1},\wt{Z}^2_{t,j},\wt{Z}^1_{t,j+1},\ldots,\wt{Z}^1_{t,d})\Big]\I_{\wt{D}_j}
\end{align*}
where $\wt{Z}_{t,j}^{i}:=(m_t Z_t^{i})_{j}$ for $i=1,2$ and $\wt{D}_j :=\{\wt{Z}^1_{t,j}\neq \wt{Z}^2_{t,j}\}$ for $j=1,2,\dots,k$.
For brevity, we denote by $c$ the $\brr^k$-valued process $c=(c^1,\dots, c^k)^*$. Recall that $\wh{g}(t,y,z):=g(t,y,m^{-1}_tz)$ and thus if $g$ is  uniformly $m$-Lipschitz continuous, then
$\wh{g}$  is uniformly Lipschitz continuous.

The following result extends the comparison theorem established in \cite{NR2016} to the case of BSDEs driven by RCLL martingales. The reader is
also referred to Theorem 3 in Dumitrescu et al. \cite{DGQS2018} and Theorem 3.2 in Peng and Xu \cite{PX2009} (see also Theorem 3.3 in Peng and Xu \cite{PX2010}) for analogous, albeit less general,
results (see also Theorem 2.2 in Carbone et al. \cite{CFS2008} for another instance of a related comparison theorem for BSDEs).
Recall that $\mu_Q$ denotes the Dol\'eans measure of the process $Q$ on the space $\Omega \times [0,T]$ endowed with the product $\sigma$-algebra
$\cG_T\otimes\cB([0,T])$ so that $\mu_Q(A)=\E\big[\int_0^T\I_A (\omega,t)\,dQ_t(\omega)\big]$ for every $A\in\cG_T\otimes\cB([0,T])$.

\begin{theorem} \label{the6.1}
Suppose that Assumptions \ref{ass2.1}, \ref{ass2.2}, \ref{ass5.1} and \ref{ass5.2} hold with an $\bff$-progressively measurable process $m$.
Let Assumption \ref{ass2.3} be satisfied by the triplet $(g^l,\eta^l,\UU^l)$ for $l=1,2$ and let $(Y^l,Z^l)$
be the unique solution to the BSDE \eqref{BSDEl}. Assume that: \hfill \break
(i) the processes $g^l(\cdot,\cdot,y,z),\, l=1,2$ are $\bff$-predictable for any fixed $(y,z)\in \brr \times \brr^{\dimn}$, \hfill \break
(ii) the process $\UU:=\UU^1-\UU^2$ is \cadlag,  nonincreasing.
\hfill \break (iii) the inequality $\eta^1\ge\eta^2$ is valid, \\
(iv) there exists a bounded process $\zeta $ 
such that $\Delta\wh{M}_t=\Delta \wh{M}^d_t> -1$ for every $t \in [0,T]$ where
\begin{equation*}
\wh{M}_t:=\int_0^t\cfun^*_u\alphat^c_u\,dM^c_u+\int_0^t\zeta^{\ast}_u\alphat^d_u\,dM^d_u=\wh{M}^c_t+\wh{M}^d_t,
\end{equation*}
and the following inequality holds
\begin{equation} \label{eq6.1}
b_t:=\delta_t-\zeta^{\ast}_t m^{d}_t(Z_t^{1,d}-Z_t^{2,d}) \ge 0 ,\quad \mu_{Q}\mbox{\rm -a.e.,}
\end{equation}
where the process $\delta $ is given by
\begin{align*}
\delta_t :=g^1(t,Y^2_t,Z^2_{t,1},\ldots,Z^2_{t,k},Z^1_{t,k+1},Z^1_{t,k+2},\ldots,Z^1_{t,d})-g^2(t,Y^2_t,Z^2_{t,1},\ldots, Z^2_{t,k},Z^2_{t,k+1},
Z^2_{t,k+2},\ldots,Z^2_{t,d}).
\end{align*}
Then the following assertions are true: \\
(a) the comparison property is valid, that is, the inequality $Y^1_t\ge Y^2_t$ holds for every $t\in[0,T]$, \\
(b) the strict comparison property holds, meaning that if $Y^1_0=Y^2_0$,
then $Y^1_t=Y^2_t$ for every $t \in [0,T]$ and thus, in particular, $\eta^1=\eta^2$.    
\end{theorem}

\proof
Since $(g^l,\eta^l,\UU^l)$ satisfies Assumption \ref{ass2.3} for $l=1,2$, then BSDE \eqref{BSDEl} has the unique solution $(Y^l,Z^l)$ such that
$(Y^l-\UU^l,Z^l)\in \SSs$. We denote $g:=g^1-g^2$ and
$$
y:=Y^1-Y^2,\ z:=Z^1-Z^2,\ \UU:=\UU^1-\UU^2,\ \eta:=\eta^1-\eta^2.
$$
It is straightforward to check that the pair $(y,z)$ satisfies $(y-\UU,z)\in \SSs$ and solves the following linear BSDE (recall the notations at the beginning of this section)
\begin{equation} \label{eq6.2}
\left\{ \begin{array}
[c]{ll}
dy_t=-\big(a_ty_t+\delta_t+\cfun^*_tm^{c}_tz^c_t\big)\,dQ_t+z_t^{\ast}\,dM_t+d\UU_t,\medskip\\
y_{T}=\eta.
\end{array} \right.
\end{equation}
Since the generator $g^1$ is assumed to be uniformly $m$-Lipschitz continuous with a constant $\wh{L}_1$, it is clear that the mapping $\wh{g}^1$
is uniformly Lipschitz continuous with the same constant. Consequently, the $\bff$-progressively measurable process $a$ and the $\bff$-predictable
process $\cfun$ appearing in \eqref{eq6.2} are bounded, specifically, $|a_t|\le \wh{L}_1$ and $\|\cfun_t \|\le \wh{L}_1\sqrt{k}$ for all $t \in [0,T]$.
Therefore, the generator in  \eqref{eq6.2} is uniformly $m$-Lipschitz continuous and thus we may apply Theorem \ref{the4.1}. Since
$b_t=\delta_t-\zeta^{\ast}_t m^{d}_tz^d_t$, we also have that
\begin{equation} \label{eq6.3}
\left\{ \begin{array}
[c]{ll}
dy_t=- \big(a_ty_t+b_t+\cfun^*_tm^{c}_tz^c_t+\zeta^{\ast}_tm^d_t z^d_t\big)\,dQ_t+z_t^{\ast}\,dM_t+d\UU_t,\medskip\\
y_{T}=\eta.
\end{array} \right.
\end{equation}
Since the generators $g^1$ and $g^2$ satisfy Assumption \ref{ass2.3} and $\zeta $ is bounded, we have $b$ belongs to $\Hs$.

In the remaining part of the proof, we focus on the linear BSDE \eqref{eq6.3}. We deduce from Proposition \ref{pro5.1} that, under the present
assumptions, if $(y,z)\in\Ss$ is a unique solution to \eqref{eq6.3}, then the process $y$ satisfies
\begin{equation} \label{eq6.4}
y_t=\qqzet_t^{-1}\,\E\bigg[\qqzet_{T}\eta+\int_t^T\qqzet_ub_u\,dQ_u-\big(\qA_T-\qA_t\big)\,\Big|\,\cG_t\bigg]
\end{equation}
where we denote
\begin{align*}
\qA_t :=\int_0^t \qqzet_{u-}\,d\UU_u+\langle \qqzet,\UU \rangle_t
\end{align*}
and the strictly positive process $\qqzet$ is given by \eqref{eq5.2}. Let us check that, under the present assumptions, the process $\qA$ is an $\bff$-supermartingale (which yields that  $\E\qA$ is nonincreasing).
Since, by assumption, the process $\UU$ is nonincreasing, which implies that $\UU$ is of finite variation, we have $[\qqzet,\UU]_t=\sum_{u\le t}\Delta\qqzet_u\Delta\UU_u$
where $\Delta \qqzet_u=\qqzet_{u-}\Delta\wh{M}_u$. Since $\Delta\wh{M}_u>-1$ and the process $\qqzet$ is strictly positive,
it is easy to see that the process $\PhiA$, which is given by
\begin{align*}
\PhiA_t:=\int_0^t\qqzet_{u-}\,d\UU_u+[\qqzet,\UU]_t=\int_0^t\qqzet_{u-}\,d\UU_u+\sum_{u\le t}\Delta\qqzet_u\Delta\UU_u
\end{align*}
is nonincreasing since the (pathwise) continuous component of $\PhiA$ is nonincreasing and its jumps satisfy
\begin{align*}
\Delta \PhiA_u=\qqzet_{u-}\Delta\UU_u+\Delta\qqzet_u\Delta\UU_u=\qqzet_{u-}\big(1+\Delta\wh{M}_u\big)\Delta\UU_u\le 0 .
\end{align*}
Since $\PhiA$ is nonincreasing, its dual $\bff$-predictable projection $\PhiA^p$ is nonincreasing as well. Moreover, it equals
\begin{align*}
\PhiA^p_t=\int_0^t\qqzet_{u-}\,d\UU^p_u+\langle\qqzet,\UU\rangle_t
\end{align*}
where $\UU^p$ is the dual $\bff$-predictable projection of $\UU $. Note that since $\UU^1_T$ and $\UU^2_T$ belong to $\LsT$, the random variable
$\UU_T$ is integrable and thus by noticing that  $\cG_0$ is trivial, the nonincreasing process $\UU$ is integrable as well and thus $\UU^p$ is well defined and the process
$M^{\UU}=\UU-\UU^p$ is a uniformly integrable $\bff$-martingale (see, for instance, Corollary 5.31 in \cite{HWY1992}). 
Moreover, from Theorem 8.2.19 in \cite{CE2015} and the property that $\UU_T$ belongs to $\LsT$, we deduce that the random variable $\UU^p_T$ is square-integrable
and thus, since $M^{\UU}_t=\E(\UU_T-\UU^p_T\,|\,\cG_t)$, we see that $M^{\UU}$ is a square-integrable $\bff$-martingale with $M^{\UU}_0=0$.
 Hence, by Assumption \ref{ass2.2}, there exists a process $X \in \Ls$ such that $M^{\UU}_t=\int_0^t X^*_u\,dM_u$. We define the process
\begin{align*}
M^{q,\UU }_t:=\int_0^t\qqzet_{u-}\,dM^{\UU}_u=\int_0^t\qqzet_{u-}X^*_u\,dM_u.
\end{align*}
Since the process $q$ belongs $\Ss$ and $X$ is in $\Ls$, we obtain
\begin{align*}
\bpp\bigg[\int_0^T\|\qqzet_{t-}m_tX^*_t\|^2\,dQ_t<\infty\bigg]\ge \bpp\bigg[\sup_{t\in[0,T]}(\qqzet^2_t)\int_0^T\|m_tX^*_t\|^2\,dQ_t<\infty\bigg]=1
\end{align*}
and thus $M^{q,\UU}$ is well defined and it is a local $\bff$-martingale. Furthermore, since the process $M^{q,\UU}$ is bounded from below by the integrable random variable
$\psi:=-\sup_{t\in[0,T]}(\qqzet_t)(|D_T-D_T^p|+|D_0-D_0^p|)$, it is in fact an $\bff$-supermartingale. To complete the proof of assertion (a), we observe that, for every $t\in[0,T]$,
\begin{align*}
\qA_t=\PhiA^p_t+\int_0^t\qqzet_{u-}\,d(\UU_u-\UU^p_u)=\PhiA^p_t+\int_0^t\qqzet_{u-}\,dM^{\UU}_u=\PhiA^p_t+M^{q,\UU }_t,
\end{align*}
which means that  the process $\qA$ is also an $\bff$-supermartingale. Since we have also assumed that $\eta=\eta^1-\eta^2\ge 0$ and the process $b\ge 0,\,\mu_Q$-a.e., it is readily seen from \eqref{eq6.4} that the inequality $y_t\ge0$ holds for all $t \in [0,T]$. Assertion (b) is also an easy consequence of representation \eqref{eq6.4}.
\endproof

\begin{remark}  \label{rem6.1}
{\rm It is easy to see that condition \eqref{eq6.1} holds if $g^1(\cdot,Y^2,Z^2)\ge g^2(\cdot,Y^2,Z^2),\,\mu_{Q}$-a.e., and
\begin{equation} \label{eq6.5}
\mu_t-\zeta_t^{\ast}m^{d}_tz^d_t\ge 0,\quad\mu_{Q}\mbox{\rm -a.e.},
\end{equation}
where we denote
\begin{align*}
\mu_t:=g^1\big(t,Y^2_t,\ldots,Z^2_{t,k},Z^1_{t,k+1},Z^1_{t,k+2},\ldots,Z^1_{t,d}\big)-g^1\big(t,Y^2_t,\ldots,Z^2_{t,k},Z^2_{t,k+1},Z^2_{t,k+2},\ldots,Z^2_{t,d}\big).
\end{align*}
In particular, if the martingale $M$ is continuous so that $k=\dimn$, then assumption (iv) in Theorem \ref{the6.1} reduces to the inequality
$g^1(\cdot,Y^2,Z^2)\ge g^2(\cdot ,Y^2,Z^2),\,\mu_{Q}$-a.e., which is the standard condition in comparison theorems for BSDEs driven by the
Brownian motion (see also Theorem 3.3 in \cite{NR2016} for the case of BSDEs driven by continuous martingales).}
\end{remark}

\begin{remark}  \label{rem6.2}
{\rm Let us assume that $M=(M^c,M^d)$ where $M^c$ is an $\brr^{\dimn-1}$-valued, continuous $\bff$-local martingale
and $M^d$ is the compensated martingale of the unit jump at an $\bff$-stopping time $\tau$ with the $\bff$-intensity process $\lambda$.
This means that we have, for every $t \in [0,T]$,
$$
M^d_t=\I_{\{\tau\le t\}}-\int_0^t(m^d_u)^2\,du=\I_{\{\tau\le t\}}-\int_0^t\lambda_u\,du
$$
so that $m^d_t=\sqrt{\lambda_t}$. Let us denote $\zeta_t=\psi_t m^d_t=\psi_t \sqrt{\lambda_t}$. Then condition $\Delta \wh{M}_t>-1$ is satisfied if $\psi_t > -1$
and inequality \eqref{eq6.5} becomes
\begin{align*}
g^1(t,Y^2_t,Z^2_{t,1},Z^1_{t,2})-g^1(t,Y^2_t,Z^2_{t,1},Z^2_{t,2})\ge\psi_t\lambda_t(Z^1_{t,2}- Z^2_{t,2}).
\end{align*}
In particular, we recover condition (25) in Dumitrescu et al. \cite{DGQS2018} and condition (12) in \cite{DQS2018}. Of course, this condition should be combined with the postulate that $g^1(\cdot,Y^2,Z^2)\ge g^2(\cdot,Y^2,Z^2),\,\mu_{Q}$-a.e..}
\end{remark}


\end{document}